\documentclass[12pt]{amsart}

\usepackage[USenglish]{babel} 
\usepackage[T1]{fontenc}
\usepackage[ansinew]{inputenc}
\usepackage{graphicx} 
\usepackage{lmodern} 
\usepackage{enumerate}
\usepackage{amsmath}
\usepackage{amsthm}
\usepackage{amsfonts}
\usepackage{amscd}
\usepackage[colorinlistoftodos]{todonotes}
\usepackage{tikz-cd} 
\usepackage{mathtools}
\usepackage{cleveref}
\usepackage{float}

\newcommand{\Hdel}{\mathcal{H}_\Delta}
\newcommand{\Hdelr}{\mathcal{H}_{0,\Delta}}
\newcommand{\Hdelc}{\overline{\mathcal{H}_\Delta}}

\newcommand{\Mtrop}{\MM^{\text{trop}}_{g,n}}
\newcommand{\PP}{\mathbb{P}}
\newcommand{\AAA}{\mathcal{A}}
\newcommand{\SSS}{\mathcal{S}}
\newcommand{\CC}{\mathbb{C}}
\newcommand{\Mesh}{\mathcal{C}}
\newcommand{\CP}{\mathbb{C}\mathbb{P}}
\newcommand{\RP}{\mathbb{R}\mathbb{P}}

\newcommand{\LL}{\mathcal{L}}
\newcommand{\NN}{\mathbb{N}}
\newcommand{\MM}{\mathcal{M}}
\newcommand{\ZZ}{\mathbb{Z}}
\newcommand{\RR}{\mathbb{R}}

\newcommand{\cG}{G}

\newcommand{\Sec}{\operatorname{Sec}}

\newcommand{\homm}{\operatorname{Hom}}
\newcommand{\Log}{\operatorname{Log}}

\newcommand{\inti}{\operatorname{int}}
\newcommand{\cone}{\operatorname{cone}}

\newcommand{\comp}{E}

\newcommand{\Conv}{\operatorname{conv}}

\newcommand{\area}{\operatorname{area}}
\newcommand{\enen}{\enspace | \enspace}

\newcommand{\pgl}{\text{PSL}_2(\RR)}

\theoremstyle{plain}

\newtheorem{theorem*}{Theorem}
\newtheorem{theorem}{Theorem}[section]
\newtheorem{proposition}[theorem]{Proposition}
\newtheorem{lemma}[theorem]{Lemma}
\newtheorem{corollary}[theorem]{Corollary}
\newtheorem{question}[theorem]{Question}
\newtheorem{conjecture}[theorem]{Conjecture}
\newcommand{\spine}{\Upsilon}

\theoremstyle{definition}
\newtheorem{definition}[theorem]{Definition}
\newtheorem{example}[theorem]{Example}

\theoremstyle{remark}
\newtheorem*{remark}{Remark}

\title{The Moduli Space of Harnack Curves in Toric Surfaces}
\author{Jorge Alberto Olarte}
\thanks{RWTH Aachen University, \textbf{email:} olarte@mathb.rwth-aachen.de}

\begin{document}

\begin{abstract}
In 2006, Kenyon and Okounkov computed the moduli space of Harnack curves of degree $d$ in $\CP^2$. We generalize to any projective toric surface some of the techniques used there. More precisely, we show that the moduli space $\Hdel$ of Harnack curves with Newton polygon $\Delta$ is diffeomorphic to $\RR^{m-3}\times\RR_{\geq0}^{n+g-m}$ where $\Delta$ has $m$ edges, $g$ interior lattice points and $n$ boundary lattice points, solving a conjecture of Cr\'etois and Lang (2018). Additionally, we use abstract tropical curves to construct a compactification of this moduli space by adding points that correspond to collections of curves that can be patchworked together to produce a curve in $\Hdel$. This compactification comes with a natural stratification with the same poset as the secondary polytope of $\Delta$.
\end{abstract}
\maketitle



\section{Introduction}

\begingroup
\renewcommand\thefootnote{}\footnote{The author was supported by the Einstein Foundation Berlin through Francisco Santos' Visiting Professor Fellowship at Freie Universit\"at Berlin.}
\renewcommand\thefootnote{}\footnote{{\bf MSC-2010:} Primary 14H50; Secondary 14H10, 14M25, 52B20,	14T05.}
\renewcommand\thefootnote{}\footnote{{\bf Keywords}: Harnack curves, amoebas, toric varieties, moduli spaces, secondary polytopes, tropical curves.}
\endgroup
Harnack curves are real algebraic plane curves inside a projective toric surface, introduced by Mikhalkin \cite{mikhalkin2000real}, with several remarkable properties. By definition, they have the maximum possible number of connected components for a given Newton polygon and these components are arranged in a unique particular way  (see Definition \ref{def:harnack}). Their amoebas 
 are particularly special, since they are precisely the ones with maximal area \cite{mikhalkin2001amoebas}.  
Because of this, they have found applications in physics, where the dimer model is used to study crystal surfaces (see \cite{kenyon2006dimers} for details). In this model, the limit of the shape of a crystal surface is given by the amoeba of a Harnack curve. 

These curves are named after Axel Harnack, who constructed them in the projective plane to show that his upper bound on the number of connected components of plane real algebraic curve is attained \cite{harnack1876ueber}. In the projective plane, the space of Harnack curves of degree $d$ modulo the action of the torus $(\CC^*)^2 \subseteq \CP^2$ was studied by Kenyon and Okounkov \cite{kenyon2006planar} to better understand the dimer model. Equivalently, this is the space of amoebas of Harnack curves modulo translation. They show that this moduli space has global coordinates given by the areas of holes of the amoeba and the distances between consecutive tentacles. Therefore it is diffeomorphic to $\RR_{\geq 0}^{(d+4)(d-1)/2}$. Cr\'etois and Lang \cite{cretois2017vanishing} generalized some of the techniques used in \cite{kenyon2006planar} to Harnack curves in any projective toric surface. They showed that given a lattice polygon $\Delta$, the moduli space $\Hdel$ of Harnack curves with Newton polygon $\Delta$ is path connected and they conjectured that it is also contractible \footnote{The conjecture appears as Remark 4.4 in the preprint version of \cite{cretois2017vanishing}; the published version already cites our results}. We confirm this belief and further generalize the results of \cite{kenyon2006planar} to compute $\Hdel$:

%

\begin{theorem*}
\label{thm1}
Let $\Delta$ be a lattice $m$-gon with $g$ interior lattice points and $n$ boundary lattice points. Then the moduli space $\Hdel$ of Harnack curves of Newton polytope $\Delta$ is diffeomorphic to $\RR^{m-3}\times\RR_{\geq0}^{n+g-m}$.
\end{theorem*}

The interior of $\Hdel$ corresponds to the smooth Harnack curves with transversal intersections with the axes of $X_\Delta$. In the boundary, ovals may contract to double points, or the curves may not intersect with the axes transversally.

We further show that $\Hdel$ admits a compactification similar in spirit to the Deligne-Mumford compactification of $\mathcal{M}_{g,n}$. This compactification consists of what we call \emph{Harnack meshes}. A Haranck mesh (see \Cref{defn:mesh}) consists of a regular subdivision of $\Delta$ and a Harnack curve with Newton polytope $\Delta_i$ for each facet $\Delta_i$ of the subdivision, with some gluing conditions that allow for the curves to be patchworked (using Viro's method, see \cite{viro2006patchworking}) to produce a curve in $\Hdel$. The space of Harnack meshes is naturally stratified into cells according to which regular subdivision is used in the patchworking recipe. The above can be summed up in the following: 

%

\begin{theorem*}
\label{thm2}
The space $\Hdel$ has a compactification $\Hdelc$ consisting of all Harnack meshes over $\Delta$. Moreover, $\Hdelc$ has a cell complex structure whose poset is isomorphic to the face poset of the secondary polytope $\Sec(\Delta\cap\ZZ^2)$. 
\end{theorem*}

The structure of the paper is as follows.
In section 2, we set notation and recall some background results on Harnack curves. Sections 3 and 4 are dedicated to proving \Cref{thm1} (\Cref{modulitheorem}) and section 5 is dedicated to prove \Cref{thm2} (\Cref{thm_compact}).

Most of the proofs consist in showing that there are different parameters that can be taken as global coordinates for Harnack curves. In section 3 we consider the following diagram:
\[
\begin{Bmatrix}
\text{Rational}\\
\text{Harnack}\\
\text{curves}
\end{Bmatrix} \hookrightarrow
\begin{Bmatrix}
\text{Roots of}\\
\text{rational}\\
\text{parametrization}
\end{Bmatrix} /\pgl
\stackrel{\tilde\rho}{\rightarrow}
\begin{Bmatrix}
\text{Positions of}\\
\text{amoeba}\\
\text{tentacles}
\end{Bmatrix} /\RR^2
\]
In the left we have the moduli space of rational Harnack curves, which we denote $\Hdelr$; in the middle we have parametrizations $\phi : \CP^1\to X_\Delta$ of Harnack curves modulo the action of $\pgl$ on $\CP^1$; and in the right we have the positions of the tentacles of the amoeba modulo translations of the amoeba. The main result of section 3 is that the map $\tilde\rho$ is a smooth embedding when restricted to the image of the first map.

In section 4 we show that the following are diffeomorphisms:
\begin{equation}
\begin{Bmatrix}
\text{Harnack curves}\\
\text{with fixed}\\
\text{tentacle positions}\\
\end{Bmatrix} \leftrightarrow
\begin{Bmatrix}
\text{Bounded}\\
\text{Ronkin}\\
\text{intercepts}
\end{Bmatrix} 
\leftrightarrow
\begin{Bmatrix}
\text{Areas of}\\
\text{holes of}\\
\text{the 	amoeba}
\end{Bmatrix}
\label{eq:fixed_boundary}
\end{equation}
By putting together the two diagrams above we have:
\[
\Hdel
\hookrightarrow
\begin{matrix}
\{\text{Tentacle positions}\}
/\RR^2 \\
\times\{\text{Bounded intercepts}\}
\end{matrix} 
\to
\begin{Bmatrix}\text{All Ronkin}\\
\text{intercepts}
\end{Bmatrix}
/\RR^3
\to
\Mtrop
\]
where the $\RR^3$ action in the the third space refers to translations of the graph of the Ronkin function. The first map is a smooth embedding by putting together the two previous diagrams.

In section 5 we look at the last two maps. The second map is a linear bijection between tentacle positions and unbounded Ronkin intercepts. The last map is given by what we call the \emph{expanded spine}.
Since $\Mtrop$ is not a manifold (it is a tropical variety), this map can no longer be a diffeomorphism. However we show that it is a piecewise linear embedding.
We also show how Harnack meshes can be similarly embedded into the closure of the embedding $\Hdel \hookrightarrow \Mtrop$, allowing us to construct the compactification $\Hdelc$. 

We end the paper by suggesting some directions for future research in section 6. In particular we conjecture $\Hdelc$ to be a CW-complex and we suggest a possible smooth structure on $\Hdelc$ as a manifold with generalized corners (see \cite{gcmanifolds}).

\subsection*{Acknowledgements}
The author would like to thank Mauricio Velasco for proposing this problem and advising the master thesis that preceded this paper as well as the rest of the committee members Tristram Bogart, Felipe Rinc\'on, Florent Schaffhauser; the organizers of the Master Class in Tropical Geometry in Stockholm 2017, where much of the inspiration for section 5 was obtained; Grigory Mikhalkin, Timo de Wolff, Lionel Lang and anonymous referees for helpful discussions and suggestions. Special thanks to Francisco Santos for all of his help and support with the development of this paper.

\section{Preliminaries}
\subsection{Notation}
We fix the following notation for the rest of the paper. As is usual in toric geometry, $M \cong N \cong \ZZ^2$ are the lattices of characters and one-parametric subgroups of the algebraic two-dimensional torus $(\CC^*)^2$ respectively. See \cite{cox2011toric} as a general reference for toric varieties.

Let $\Delta\subset M \otimes \RR$ be a convex lattice polygon. We write $\partial \Delta$ for the boundary of $\Delta$, $\inti(\Delta)$ for the interior of $\Delta$. We write $\Delta_M$ for the lattice points in $\Delta$, that is, $\Delta_M = \Delta \cap M$.
We use $n$ and $g$ to denote the number of lattice points in $\partial\Delta$ and $\inti(\Delta)$, respectively, and $m$ for the number of edges of $\Delta$. 
For any positive integer $k$, $[k]$ denotes the set $\{1,\dots,k\}$.
We denote by $\Gamma_i$, $i=1,\dots,m$, the edges of $\Delta$ in cyclic anticlockwise order. 
Let $d_1,\dots, d_m$ be their respective integer lengths (i.e. $d_i = |\Gamma_i\cap M| -1$). Let $u_i \in N$ be the primitive inner normal vector of $\Gamma_i$.
We have the following equation:
\begin{equation}
\sum\limits_{i=1}^m d_ku_i = 0
\label{eq:balancing}
\end{equation}

To each $v=(v_1,v_2)\in M$ there is an associated Laurent monomial $x^v := x_1^{v_1} x_2^{v_2}$. The \emph{Newton polygon} of a Laurent polynomial $f(x) = \sum \limits_{v\in M} c_vx^v$ is the convex hull of $\{v\in M \enen c_v \ne 0\}$. For any subset $\Delta' \subseteq \Delta$ we write $f|_{\Delta'}(x) :=  \sum \limits_{v\in \Delta_M'} c_vx^v$.

Given a lattice polygon $\Delta$ there is an associated projective toric surface $X_\Delta$ whose geometry reflects the combinatorics of $\Delta$. It contains a dense copy of the torus $(\CC^*)^2$ where coordinate-wise multiplication extends to an action on all of $X_\Delta$. For each edge $\Gamma_i$ of $\Delta$, there is a corresponding irreducible divisor $L_i$ in $X_\Delta$ which is invariant under the action of the torus. We call these divisors the \emph{axes} of $X_\Delta$. Two axes intersect in a point if and only if they correspond to consecutive edges of $\Delta$. We denote the real part of $X_\Delta$ as $\RR X_\Delta$.



\subsection{Harnack Curves}

Let $f$ be a Laurent polynomial with real coefficients and Newton polygon $\Delta$.
The zeros of $f$ define a curve $C^\circ \subset (\CC^*)^2$. The closure of $C^\circ$ in $X_\Delta$ is a compact algebraic curve $C$. If $C$ is smooth its genus is equal to $g$ \cite{khovanskii1978newton}. The intersection of $C$ with $\RR X_\Delta$ is a real algebraic curve $\RR C$. The intersection of $C$ with an axis $L_i$ is given by the restriction of $f$ to $\Gamma_i$, which is, after a suitable change of variable, a polynomial of degree $d_i$. Therefore $L_i \cap C$ consists of exactly $d_i$ points counted with multiplicities. 

\begin{definition}{\cite[Definition 2]{mikhalkin2000real}}
\label{def:harnack}
Let $\Delta$ be a lattice polygon with $g, m$ and the $d_i$'s defined as above. A smooth real algebraic curve $\RR C\subseteq \RR X_\Delta$ is called a \emph{smooth Harnack curve} if the following conditions hold:
\begin{itemize}
	\item The number of connected components of $\RR C$ is $g +1$.
	\item Only one component of $\RR C$ intersects $L_1\cup\dots\cup L_m$. This component can be subdivided into $m$ disjoint arcs, $\theta_1 \dots \theta_m$, in that order, such that $C\cap L_i = \theta_i \cap L_i$.
\end{itemize}
The components that are disjoint from $L_1\cup\dots\cup L_m$ are called \emph{ovals}.
\end{definition}
Harnack curves were originally called in ``cyclically maximal position'' in \cite{mikhalkin2000real}.
In the literature these curves are sometimes called ``simple Harnack curves''. 
However, following \cite{mikhalkin2001amoebas,kenyon2006dimers,kenyon2006planar} we omit the adjective `simple' when referring to them (see \cite[Remark 6.6]{mikhalkin2007geometry}).

These curves are named after Axel Harnack because he showed in 1876 that smooth curves of genus $g$ in the real projective plane have at most $g+1$ connected components. To show that the bound was tight, he constructed the eponymous curves \cite{harnack1876ueber}. Curves which attain the maximum number of components are called $M$-curves. These are the topic of the first part of Hilbert's 16$^{th}$ problem, which asks to classify all possible topological types of $M$-curves. When $\RR X_\Delta = \RP^2$, Harnack curves are the $M$-curves such that only one component intersects the axes and it does so in order. 
Mikhalkin proved that, for any given $\Delta$, if $\RR C$ is a Harnack curve with Newton polygon $\Delta$ then the topological type of $(\RR X_\Delta, \RR C, \RR L_1\cup,\dots,\cup \RR L_n)$ is unique \cite[Theorem 3]{mikhalkin2000real}. 

Recall that a singular point in $\RR C$ is an ordinary isolated double point if it is locally isomorphic to the singularity of $x_1^2+x_2^2 =0$.
\begin{definition}{\cite[Definition 3]{mikhalkin2001amoebas}}
A (possibly singular) real algebraic curve $\RR C\subseteq X_\Delta$ is a \emph{Harnack curve} if 
\begin{itemize}
	\item The only singularities of $\RR C$ are ordinary isolated double points away from the torus invariant divisors.
	\item Replacing each singular point of $\RR C$ by a small oval around it yields a curve $\RR C'$ such that $(\RR X_\Delta, \RR C', \RR L_1\cup,\dots,\cup \RR L_n)$ has the topological type of smooth Harnack curves.
\end{itemize}
\end{definition}
Notice that any singular Harnack curve can be approximated by smooth Harnack curves. To see this, let $f$ be a polynomial that vanishes on $\RR C$ and let $g(x,y) := f(\lambda x, \lambda y)$ for a real number $\lambda$ close to 1 but different from 1, so that the singular points of $f$ and $g$ are close but do not coincide. Then $f-\epsilon g$ vanishes on a smooth Harnack curve which approaches $\RR C$ when $\epsilon$ tends to 0.

Let $\RR[\Delta_M]$ be the vector space of real polynomials with Newton polygon contained in $\Delta$. Since scaling all coefficients of $f$ by the same constant does not change the curve $\RR C$, we can identify the space of real curves with Newton polygon contained in $\Delta$ with $\PP(\RR[\Delta_M])$. The action of the torus $(\RR^*)^2$ on $\RR X_\Delta$ induces an action on $\PP(\RR[\Delta_M])$ given by $f(x_1,x_2) \mapsto f(r_1^{-1} x_1, r_2^{-1}x_2)$.

\begin{definition}
The \emph{moduli space $\Hdel$ of Harnack curves} is the subspace of $\PP(\RR[\Delta_M])/(\RR^*)^2$ consisting of all (possibly singular) Harnack curves with Newton polygon $\Delta$ modulo the action of $(\RR^*)^2$. 
\end{definition}
Given an element in $\RR C \in \Hdel$, we say that a polynomial vanishes on $\RR C$ if its zero locus is in the equivalence class given by $\RR C$.

\begin{remark}
The notation $\Hdel$ was used in \cite{cretois2017vanishing} to note the space of Harnack curves without taking them modulo the action of $(\RR^*)^2$. They defined it with a more algebro-geometric language as follows: the space of curves with Newton polygon contained in $\Delta$ can be identified with the complete linear system $|D_\Delta|$ of the Cartier divisor $D_\Delta$ of $X_\Delta$ associated to $\Delta$. Since $X_\Delta$ is a complete normal toric variety, $|D_\Delta|$ can be identified with the projectivization of the space of global sections of the line bundle associated to $\Delta$. Therefore $\Hdel$ can be defined as the subspace of $|D_\Delta|$ of Harnack curves, modulo the action of the torus $(\RR^*)^2$ on $\RR X_\Delta$.
\end{remark}

The case when $\Delta$ is the $d$-th dilation of the unimodular triangle corresponds to degree $d$ curves in $\RP^2$ and the moduli space is diffeomorphic to the closed orthant $\RR_{\ge 0}^{(d+4)(d-1)/2}$ \cite[Corollary 11]{kenyon2006planar}. 

\subsection{Amoebas and the Ronkin function}

Amoebas, which are essential to understand Harnack curves, were defined in \cite[Chapter 6]{gelfand2008discriminants} where details about them can be found.
\begin{definition}
Let $\Log: (\CC^*)^2 \rightarrow \RR^2$ be the map
\[
\Log (z_1,z_2) := (\log|z_1|, \log|z_2|)
\]
The \emph{amoeba} of an algeraic curve $C$ is $\mathcal{A}(C) := \Log(C^\circ)$. 
\end{definition}

The amoebas of Harnack curves are specially well-behaved:
\begin{proposition}{\cite{mikhalkin2001amoebas}}
Let $\RR C$ be a real algebraic curve with Newton polygon $\Delta$ and $\AAA= \AAA(C)$ its amoeba. The following are equivalent:
\begin{enumerate}	
	\item $\RR C$ is Harnack curve
	\item The map $\Log|_{C^\circ}$ is at most 2-to-1.
	\item $\area(\AAA) = \pi^2\area(\Delta)$
\end{enumerate}
\end{proposition}
For arbitrary curves, $\area(\AAA) \leq \pi^2\area(\Delta)$ \cite{passare2004amoebas}, so Harnack curves have the amoebas with maximal area. Smooth Harnack curves are also characterized by having maximal curvature, and by having totally real logarithmic Gauss map \cite{passare2010curvature,mikhalkin2000real}. However there are more general singular curves whose logarithmic Gauss map is also totally real \cite{lang2015generalization}.


Each connected component of the complement of an amoeba is convex and has a point in $\Delta_M$ naturally associated to it, as we now show.
Let $f: \RR^2 \rightarrow \RR$ be a Laurent polynomial. The \emph{Ronkin function} $R_f : \RR^2 \rightarrow \RR$ of $f$ defined in \cite{Ronkin}, is
\[
R_f(x) := \dfrac{1}{(2\pi \sqrt{-1})^2}\int\limits_{\Log^{-1}(x)}\dfrac{\log|f(z_1,z_2)|}{z_1z_2}dz_1 dz_2.
\]

The Ronkin function is convex, see \cite{passare2004amoebas}. Its gradient vector $\nabla R_f = (\nu_1,\nu_2)$ is given by
\[
\nu_i(x) = \dfrac{1}{(2\pi \sqrt{-1})^2}\int\limits_{\Log^{-1}(x)}\dfrac{z_i\partial_{z_i} f(z_1,z_2)}{z_1z_2f(z_1,z_2)}dz_1 dz_2.
\]

For any $x\in \RR^2$ we have that $\nabla R_f(x) \in \Delta$. If two points are in the same connected component of $\RR^2 \setminus \AAA$, then their preimages under $\Log$ are homologous cycles in $(\CC^*)^2\setminus C^\circ$. This implies that $\nabla R_f$ is constant in each component and it has integer coordinates by the residue theorem. Therefore $\nabla R_f(x)$ induces an injection from the components of $\RR^2 \setminus \AAA$ to $\Delta_M$. The value that $\nabla R_f$ takes in a component of $\RR^2 \setminus \AAA$ is called the \emph{order} of that component and we write $\comp_v$ for the component of order $v$ if it exists. For details of this construction see \cite{forsberg2000laurent}. 

To better understand amoebas, we review some facts about their behaviour, see \cite[Section 6.1]{gelfand2008discriminants}. The component $\comp_v$ is bounded if and only if $v$ is in the interior of $\Delta$.
For each vertex $v$ of $\Delta$, $\comp_v$ exists and contains a translation of $-\cone(u_i,u_{i+1})$ where $u_i$ and $u_{i+1}$ are the inner normal vectors of the edges adjacent to $v$. 
If $v$ is a lattice point in the relative interior of an edge $\Gamma_i$, $\comp_v$ is only unbounded in the direction $-u_i$.
Parts of the amoeba extend to infinity in between the unbounded components of $\RR^2 \setminus \AAA$, in direction $u_i$ for some $i$. These are called the \emph{tentacles} of the amoeba. Figure \ref{fig_amoeba} serves as an illustration of how typical amoebas of Harnack curves look like.



	
	
For each $v \in \Delta_M$ such that $\comp_v$ exists, let $F_v: \RR^2 \rightarrow \RR$ be the affine linear function that coincides with $R_f$ in $\comp_v$. 
The  \emph{spine} of a curve $C$ as defined in \cite{passare2004amoebas} is the corner locus of $\max F_v$ where $\max$ is taken over all $\comp_v$ that exist. Notice that scaling $f$ by a a constant only changes $R_f$ by an additive constant constant, so the spine of $C$ is well defined. 

The spine varies continuously for smooth curves. However, if $\comp_v$ vanishes for $v\in \inti(\Delta)$, then the spine changes abruptly. 
Fortunately, for Harnack curves there is an easy work around. By the definition of singular Harnack curves, for each $v\in \inti(\Delta)\cap M$ such that $\comp_v = \emptyset$ there is an isolated double point $p_v$ in $\RR C$ such that there is a smooth Harnack curve $\RR C'$ arbitrarily close to $\RR C$ with a component near $p_v$ with order $v$. Therefore $\nabla R_f(\Log(p_v))= v$. Let $F_v$ be the tangent plane of $R_f$ at $\Log(p_v)$.
%

\begin{definition}
\label{def_spine}
Let $\RR C$ be a Harnack curve. We call \emph{expanded spine} of $C$ and denote $\spine(C)$ the corner locus of the piecewise affine linear convex function $\max \limits_{v\in\Delta_M} F_v$.
\end{definition}

The expanded spine and the usual spine coincide if and only if $\RR C$ is a smooth Harnack curve. The expanded spine varies continuously for Harnack curves, even singular ones. It has a cycle for each $v\in \inti(\Delta)\cap M$. The bounded part of the expanded spine is a planar graph of genus $g$. This definition will be crucial in \Cref{sec:compact}.

\begin{figure}[H]
\centering	
	\hspace{5mm}
	  \begin{minipage}[b]{0.42\textwidth}
    \includegraphics[width=0.9\textwidth]{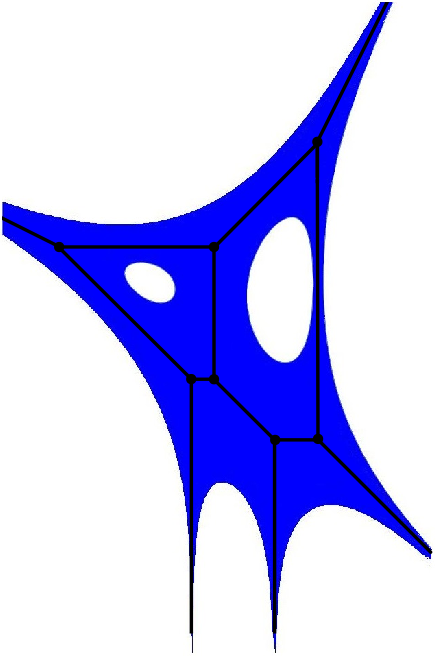}
	\end{minipage}
  \hfill
  \begin{minipage}[b]{0.45\textwidth}
	\scalebox{1}{
    \includegraphics[width=0.9\textwidth]{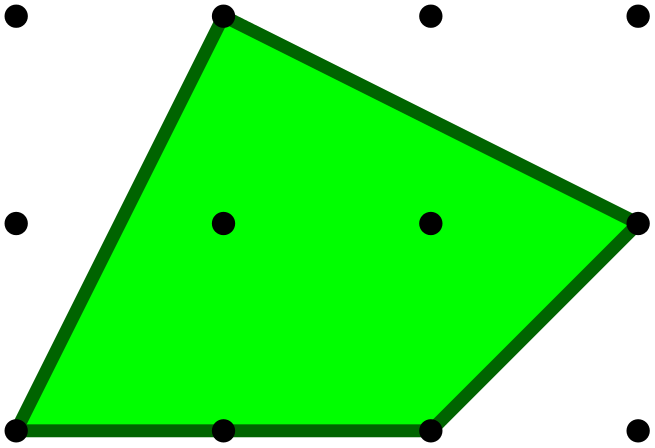}}
		\vspace{1.5cm}
  \end{minipage}
	\hspace{5mm}
	\caption{The amoeba (blue), the spine (black) and the Newton polygon (green) of a Harnack curve.}
	\label{fig_amoeba}
\end{figure}

By definition, the expanded spine is a tropical plane curve. The intercepts $c_v$ of the affine functions $F_v$ are the coefficients of the tropical polynomial that vanishes on the expanded spine, that is 
\[
\spine(C) = \operatorname{trop}(\bigoplus\limits_v c_v\odot x^{\odot v}).
\]
We call the numbers $c_v$ the \emph{Ronkin intercepts}. We call a Ronkin intercept $c_v$ \emph{bounded} if $v\in \inti (\Delta)$ and \emph{unbounded} if $v\in \partial \Delta$. In other words, we say $c_v$ is bounded if and only if whenever $E_v$ exists it is bounded.



\subsection{Patchworking}
We now give a basic overview of patchworking of real algebraic curves, a powerful tool to construct curves with a prescribed topology, developed by Viro, see \cite{viro2006patchworking} for details. 

Patchworking makes use of regular polyhedral subdivisions, which we now briefly review. We refer the reader to \cite{de2010triangulations} for more about regular subdivisions. Consider a function $h:\Delta_M \rightarrow \RR$, which we call a \emph{height function}. 
The $3$-dimensional polyhedron
\[P= \Conv(\{(v,t) \enen v \in \Delta_M \enspace  t\geq h(v)\})\] 
is only unbounded in the $(0,0,1)$ direction. For any face $F$ of $P$, let $B_F = \{v\in \Delta_M \enen (v,h(v))\in F\}$ be the collection of points in $\Delta_M$ which lift to $F$. The \emph{regular subdivision} of $\Delta_M$ induced by $h$ is $\SSS(h) := \{B_F \enen F \text{ face of } R\}$. Notice that we treat subdivisions as a collection of subsets of $\Delta_M$, such that their convex hulls are a subdivision of $\Delta$ in the usual sense.

We call a set $B\in \SSS$ a \emph{facet} of $\SSS$ if $\Conv(B)$ is two dimensional. Given a regular subdivision $\SSS$ consider the cone $\sigma(\SSS) = \{h\in \RR^{\Delta_M} \enen \SSS = \SSS(h)\}$. If $q$ is the restriction of any affine function $\RR^2 \to \RR$ to $\Delta_M$, then $\SSS(h+q) = \SSS(h)$. This implies that this fan has a linearity space of dimension 3. The collection of all cones $\sigma(\SSS)$ is a complete fan in $\RR^{\Delta_M}$ called the \emph{secondary fan} of $\Delta_M$, see \cite[Chapter 7]{gelfand2008discriminants}. The secondary fan happens to be the normal fan of a polytope in $\RR^{\Delta_M}$ called the \emph{secondary polytope} of $\Delta_M$. It is of dimension $\Delta_M -3 = n+g-3$.

\medskip

The ingredients for (real) pathworking are a regular subdivision $\SSS = \SSS(h)$ of a polygon $\Delta$ and a real polynomial $f\in \RR[\Delta_M]$. Let $\Delta_1,\dots,\Delta_s$ be the facts of $\SSS$. Then the polynomial $f|_{\Delta_i}$ defines a real curve $\RR C_i \subseteq X_{\Delta_i}$. Suppose that every curve $\RR C_i$ is smooth and intersects transversally the axes of $X_{\Delta_i}$, that is, $\RR C_i$ intersects each axis in $d$ different points, where $d$ is the integer length of the corresponding edge of $\Delta_i$.
%
Let
\[
f_t(x) := \sum \limits_{v\in \Delta_M} t^{h(v)}a_vx^v
\]
and let $\RR C_t\subseteq X_\Delta$ be the vanishing locus of $f_t$. 
The Patchworking Theorem by Viro \cite{viro2006patchworking} says that there exists $t_0 >0$ small enough such that for every $t\in (0,t_0]$ the topological type of $\RR C_t$ can be computed from the topological type of each $\RR C_i$ by gluing them in a certain way. 
We say $\RR C_t$ is the result of patchworking the curves $\RR C_i$. Given the curves $\RR C_i$, there exist different $f$ such that $f|_{\Delta_i}$ vanishes on $\RR C_i$. 
However, the topological type of the resulting curve $\RR C_t$ only depends on the signs of each $f|_{\Delta_i}$ (real polynomials with the same zero locus differ only by scaling by a constant in $\RR^*$).

We do not show in general how to do this computation, see \cite{viro2006patchworking} for that purpose. 
We do however mention some important facts regarding Harnack curves. First, Mikhalkin showed that Harnack curves can be constructed using patchworking, \cite[Appendix]{mikhalkin2000real}. There it is shown that Harnack curves are \emph{$T$-curves}, that is, curves whose topological type can be obained from patchworking using regular unimodular triangulations as regular subdivision. In that case, the signs of each coefficient of $f$ contain all the relevant information and this is known as combinatorial pathcworking \cite{itenberg1996patchworking}. Consider the sign configuration $\Delta_M \to \{-1,1\}$ given by $v\mapsto (-1)^{v_1v_2}$. No matter the triangulation chosen, the result from patchworking with this sign configuration will always be a Harnack curve \cite[Apendix]{mikhalkin2000real}. Moreover, it is essentially (up to $\ZZ_2^{2}$) the only sign configuration whose patchwork is invariant under the chosen unimodular triangulation. These statements follow directly from the discussion in \cite[Chapter 11 Section 5C]{gelfand2008discriminants}.

Another important fact is that for any regular subdivision $\SSS$, if each curve $\RR C_i$ is a Harnack curve, then there exists a choice of $f_i$ such that the result from patchworking is a Harnack curve, see \Cref{harnack_patchwork}.

\subsection{Cox coordinates}

We now review Cox coordinates for toric surfaces, since it will help us understand the parametrizations of rational Harnack curves. For details see Chapter 5 of \cite{cox2011toric}. They are a generalization of homogeneous coordinates in the projective space $\CP^d = (\CC^{d+1}\setminus\{0\})/\CC^*$. 

Let $\Delta$ be a Newton polygon and recall $u_1,\dots, u_m$ to be the primitive inner normal vectors of $\Delta$. Let $\alpha: (\CC^*)^m \to (\CC^*)^2$ be the group homomorphism given:
\[
(z_1,\dots,z_m) \mapsto (\prod_{i=1}^m z_i^{u_{i1}}, \prod_{i=1}^m z_i^{u_{i2}}).
\]
We have that $\ker(\alpha)$ is a subgroup of $(\CC^*)^m$. Let $Z$ be the subset of $\CC^m$ with at least three coordinates equal to 0 or at least two not-cyclically-consecutive coordinates equal to 0.


%
\begin{proposition}{\cite[Theorem 5.1.11]{cox2011toric}}
Let $\Delta$ be a lattice polygon. All $\ker(\alpha)$-orbits of $\CC^m \setminus Z$ are closed and the quotient $(\CC^m \setminus Z)/\ker(\alpha)$ is isomorphic to $X_\Delta$ as an algebraic variety.
\end{proposition}
We write $[z_1:\cdots: z_m]_\Delta$ to denote the point in $X_\Delta$ corresponding to the orbit of $(z_1,\dots,z_m)\in \CC^m \backslash Z$ under the action of $\ker(\alpha)$. 
We have that
\[
L_i = \left\{[z_1:\cdots: z_m]_\Delta \in X_\Delta \enen z_i = 0\right \}
\]
and
\[
\RR X_\Delta = \left\{[z_1:\cdots: z_m]_\Delta \in X_\Delta \enen z_i\in \RR \enspace \forall i \right \}.
\]

\begin{example}
Let $\Delta$ be any rectangle with edges parallel to the $\RR^2$ axes. The map $\alpha: M \rightarrow \ZZ^4$ is given by the matrix:
\[
\begin{pmatrix}
	1 & 0 & -1 &0\\
	0 &  1 & 0 & -1
\end{pmatrix}
\]
Then $\alpha: (\CC^*)^4 \rightarrow (\CC^*)^2$ is given by $(z_1,z_2,z_3,z_1) \mapsto (z_1z_3^{-1}, z_2z_1^{-1})$. The action of $G$ consists of coordinatewise multiplications by vectors of the form $(\lambda_1,\lambda_2,\lambda_1,\lambda_2)$ where $\lambda_1,\lambda_2 \in \CC^*$. The set $Z$ consists of the points where $z_1=z_3=0$ or $z_2=z_1=0$. So in this case $X_\Delta$ is isomorphic to $\CP^1 \times \CP^1$.
\end{example}

%

\section{Rational Harnack curves}

\subsection{Parametrizations of rational Harnack curves}
We start this section by describing a parametrization of rational Harnack curves which was already used in \cite{kenyon2006planar} for $X_\Delta = \CP^2$ and more generally in \cite{cretois2017vanishing}. We rewrite it using Cox homogenous coordinates. Real rational curves in $X_\Delta$ with Newton polygon $\Delta$
 can be parametrized by $\phi = [p_1:\dots:p_m]_\Delta$, where each $p_i: \CP^1 \rightarrow \CC$ is a homogenuous polynomial of degree $d_i$ with real coefficients for $i\in[m]$ and no two polynomials have a common root.
If the curve is Harnack, $\phi(\RP^1)$ is the 1-dimensional component of $\RR C$. This implies that the roots of $p_i$ are real and are ordered in the cyclic way according to \Cref{def:harnack}. In fact, this condition is sufficient for $\RR C$ to be a Harnack curve. This was shown in \cite[Proposition 4]{kenyon2006planar} for $X_\Delta = \CP^2$ and it was noticed in \cite[Equation (2)]{cretois2017vanishing} that the same arguments work for any projective toric surface. So $\RR \phi(\CP^1)$ is a Harnack curve if and only if, for some chart of $\CP^1$, we have
\begin{equation}
\label{rationalparam}
\phi(t) = \left[b_1\prod\limits_{i=1}^{d_1}(t-a_{1,i}): \enspace \cdots \enspace: b_m\prod\limits_{i=1}^{d_m}(t-a_{m,i})\right]_\Delta
\end{equation}
where all $a_{i,j}$ are real, all $b_i$ are real different from zero and
\begin{equation}
\label{eq:roots}
a_{1,1} \le \dots \le a_{1,d_1} < a_{2,1} \le \dots\le a_{2,d_n} <\dots < a_{m,1} \le \dots \le a_{m,d_m}.
\end{equation}
We call the $a_{1,1}\dots a_{m,d_m}$ the \emph{roots} of $\phi$. 
Composing $\phi$ with $\alpha$ yields a parametrization for $C^\circ \subset (\CC^*)^2$:

\[
\alpha\circ \phi(t) = \left(\prod\limits_{i=1}^m b_i^{u_{i1}}\prod\limits_{j=1}^{d_i}(t-a_{i,j})^{u_{i1}}  ,\prod\limits_{i=1}^m b_i^{u_{i2}}\prod\limits_{j=1}^{d_i}(t-a_{i,j})^{u_{i2}} \right)
\]

Let $\Hdelr$ be the subspace of $\Hdel$ consisting of rational Harnack curves.
The following generalizes \cite[Corollary 5]{kenyon2006planar}:

\begin{proposition}
\label{rationalmoduli}
Let $\Delta$ be a lattice polygon with $m$ sides and $n$ lattice points in its boundary. Then $\Hdelr$ is diffeomorphic to $\RR^{m-3}\times\RR_{\geq0}^{n-m}$. 
\end{proposition}
\begin{proof}
The parametrization above is unique up to the action of the projective special linear group $\pgl$ on the parameter $t$. This induces an action of $\pgl$ on the roots of $\phi$. More explicitely, for $\psi \in \pgl$ the function $\phi \circ \psi^{-1}$ also parametrizes $C$ and has roots $\psi(a_{1,1}),\dots, \psi(a_{m,d_m})$. The roots $\psi(a_{1,1}),\dots, \psi(a_{m,d_m})$ are in the same cyclic order as in \Cref{eq:roots}. 
 
The action of $(\RR^*)^2$ in $\RR X_\Delta$ affects $\phi$ by changing the constants $b_1,\dots,b_n$, but not the roots. The same is true for choosing different representatives of the Cox homogenous coordinates of $X_\Delta$ in \cref{rationalparam}. So every rational Harnack curve is equivalent in $\Hdel$ to a curve with $b_1=\dots=b_d = 1$. Therefore, the elements of $\Hdel$ corresponding to rational curves are uniquely determined by the roots $a_{1,1}\dots a_{m,d_m}$, up to the action of $\pgl$ on them. 

The action of $\pgl$ can be fixed, for example, by considering the unique M\"obius transformation $\psi\in \pgl$ such that $\psi(a_{1,1})= 0$, $\psi(a_{2,1})= 1$ and $\psi(a_{3,1})= 2$. The map
\begin{align*}
\RR C \mapsto & \hspace{5mm} (\psi(a_{4,1}),\dots,\psi(a_{m,1}))\\
&\times(\psi(a_{1,2})-\psi(a_{1,1}),\dots,\psi(a_{m,d_m})-\psi(a_{m,d_m-1}))
\end{align*}
is a diffeomorphism between $ \Hdelr$ and $\RR^{m-3}\times\RR_{\geq0}^{n-m}$ (by identifying copies of $\RR$ with $(\phi(a_{i,d_i}),0)$ for $i\ge 3$). 
\end{proof}

The roots of $\phi$ are associated naturally to the segments of primitive segments of $\partial \Delta$. The proposition above says that we can take as global coordinates of $\Hdelr$ 	the difference between two consecutive roots corresponding to the same edge of $\Delta$ together with the first root of each edge except the first 3 edges. 

\subsection{The positions of the tentacles}

Following \cite{kenyon2006planar}, we now make a useful change of global coordinates in $\Hdelr$. Instead of the roots of $\phi$, we will use the positions of the tentacles of the amoeba which correspond to boundary points $C\setminus C^\circ$. 

Let $J:N\rightarrow M$ be the $2\times 2$ matrix that rotates vectors $\pi/2$ clockwise, i.e. $J=\left(\begin{smallmatrix} 0 &1\\ -1 & 0 \end{smallmatrix}\right)$. Observe that $Ju_i$ is a character, which maps $[x_1,\dots,x_m]_\Delta \mapsto \prod_{k=1}^m x_k^{u_k\wedge u_i}$ where $u_j\wedge u_i \in \RR$ is the determinant of the $2\times2$ matrix whose columns are $u_j$ and $u_i$ in that order. In other words, $u_i\wedge u_j = \langle J u_j, u_i\rangle$. This is not well-defined over all $X_\Delta$. We take $0^0=1$ by convention, so $Ju_i$ is well-defined over $L_i$ except the torus invariant points. In fact, the ring of functions of $L_i$ (without the torus invariant points) consists of Laurent polynomials on $Ju_i$.

\begin{definition}
Let $\RR C$ be a rational Harnack curve parametrized by $\phi$ as in \Cref{rationalparam}. For $1\le i\le m$ and $1\le j\le d_i$, \emph{the position of the $(i,j)$ tentacle of the amoeba} is
\[
\log|\phi(a_{i,j})^{Ju_i}|
\]
\end{definition}
Explicitly,
\[
\log|\phi(a_{i,j})^{Ju_i}| = \sum\limits_{k\ne i}^m \sum\limits_{l=1}^{d_k} u_k\wedge u_i (\log|b_i| + \log|a_{i,j}-a_{k,l}|).
\]
Since the curve does not change when acting the roots by $\pgl$, the positions of the tentacles of the amoeba do not change either. However, the tentacle positions are not invariant under the action of $(\RR^*)^2$ on $\RR X_\Delta$. Concretely, multiplying $\RR C$ by $r \in (\RR^*)^2$ translates the amoeba by $\Log|r|$, thus changing the position of the $(i,j)$ tentacle by $\langle Ju_i, \Log |r| \rangle$. Thus, the $(\RR^*)^2$ action on $\RR X_\Delta$ induces an $\RR^2$ action on the position of the tentacles by translations of the amoeba.  

Consider the maps $\rho_{i,j}: \RR^n \dashrightarrow \RR$ given by
\[
\rho_{i,j} := \sum\limits_{k\ne i}^n \sum\limits_{l=1}^{d_k} u_k\wedge u_i \log|a_{i,j}-a_{k,l}|.
\]

This is almost the position of the $(i,j)$ tentacle except that we drop the $\log|b_i|$ terms. The maps $\rho_{i,j}$ are well-defined on the space of roots satisfying \cref{eq:roots}. Together they form a map $\rho: \RR^n \dashrightarrow \RR^n$ invariant under the $\pgl$ action on the roots. Additionally, consider the action of $\RR^2$ on the target space $\RR^n$ of $\rho$ given by translations of the amoeba, that is:

\begin{equation}
r\cdot \rho_{i,j} := \rho_{i,j}+\langle Ju_i, r \rangle
\label{eq:tentacles_action}
\end{equation}

By construction, $\rho$ descends to a map $\tilde{\rho}$ making the following diagram commutative:
\[ \begin{tikzcd}
\RR^n \arrow[dashrightarrow ]{r}{\rho} \arrow[swap]{d}{} & \RR^n \arrow{d}{} \\%
\RR^n /\pgl  \arrow[dashrightarrow ]{r}{\tilde{\rho}}& \RR^{n-2}
\end{tikzcd}
\]
The left downwards arrow is the quotient of the space of roots by $\pgl$ and the right downward arrow is the quotient of $\RR^n$ by the action of $\RR^2$ described above.
By \Cref{rationalmoduli}, we can identify $\Hdelr$ with the space of roots that satisfy \cref{eq:roots} modulo the $\pgl$ action, so $\tilde{\rho}$ is a well defined map on $\Hdelr$. 
The main objective of this section is to prove the following:
\begin{theorem}
\label{theodiffeo}
The restriction $\tilde{\rho}|_{\Hdelr}$ is a smooth embedding.
\end{theorem}
This is a generalization of Theorem 4 in \cite{kenyon2006planar}, where the same statement is proven for the case where $\Delta$ is a dilated unit triangle. To prove that $\tilde{\rho}|_{\Hdelr}$ is a diffeomorphism we show that it is proper (\Cref{subsec:proper}) and that the differential is injective (\Cref{subsec:jacobian}). 

Before going to the proof, let us show a concrete diffeomorphism between the positions of the tentacles and $\RR^{m-3}\times\RR_{\ge0}^{n-m}$. For the semi-bounded components, consider the distance between two parallel tentacles
\[
\rho_{i,j+1}-\rho_{i,j} \enspace \text{ for } 1\le i \le m, \enspace 1\le j \le d_i-1.
\]
 Notice that $\rho_{i,j+1}-\rho_{i,j}$ is invariant under the $\RR^2$-action of translating the amoeba. For the unbounded components take as coordinates
\[
\tilde{\rho}_{i,1} \enspace  \text{ for } 4 \le i \le m,
\]
where $\tilde{\rho}_{i,1}$ is the position of the $(i,1)$-tentacle after translating the amoeba, so that the position of the $(1,1)$ and $(2,1)$ tentacles are both 0. It is straightforward to see that the image satisfies
\begin{equation}
\sum_{i=1}^m \sum_{j=1}^{d_i} \rho_{i,j} = 0 \quad \text{ and }\quad \rho_{i,j+1}-\rho_{i,j}\ge 0 \quad \text{ for all $i,j$}.
\label{eq:codomain}
\end{equation}
In particular, $\tilde{\rho}_{3,1}$ is determined by the rest of the coordinates. 
The space described by the equation and the inequalities above is simply connected. This is important since we will use the following global diffeomorphism theorem, which was known by Hadamard. A proof of it can be found in \cite{Gordon}.
\begin{proposition}
\label{prop_hadamard}
A local diffeomorphism between two manifolds which is proper and such that the image is simply connected is a diffeomorphism.
\end{proposition}

 
\subsection{Properness}
\label{subsec:proper}
To prove that $\tilde{\rho}|_{\Hdelr}$ is proper, we make use of the following lemma:
\begin{lemma}
\label{lemma:proper}
Let $x_1,\dots x_n,y_1,\dots, y_n \in \RR$ such that $x_1\le \dots \le x_n$, $\sum\limits_{i=1}^n y_i = 0$ and there exists $j$ such that $y_i <0$ for $i<j$ and $y_i \ge 0$ for $i> j$. Then 
\[
\sum_{i=1}^n x_iy_i \ge 0.
\]
\end{lemma}
\begin{proof}
We do induction on $n$. For $n=1$ we have $y_1= 0$ so the above sum is $0$. Let $n>1$ and suppose the inequality above holds for less than $n$ terms. Subtract from the left-hand-side $y_n(x_n-x_{n-1})$, which is a non-negative number. The result is
\[
\sum_{i=1}^{n-2} x_iy_i +x_{n-1}(y_{n-1}+y_n)
\] 
which is non-negative by applying the induction hypothesis to $x_1, \dots, \allowbreak x_{n-1}, y_1, \dots, y_{n-2}, y_{n-1}+y_n$. 
\end{proof}

\begin{proposition}
The map $\tilde{\rho}|_{\Hdelr}$ is proper when restricting the co-domain to the quotient of the space described by the equation and inequalities in (\ref{eq:codomain}).  
\label{properpropo}
\end{proposition}
\begin{proof}
We want to show that the preimage of a compact set is a bounded set in $\Hdelr \cong \RR^{m-3}\times\RR_{\ge0}^{n-m}$. First we bound the parameters corresponding to the $\RR_{\ge 0}^{n-m}$ part of $\Hdelr$. These correspond to the difference between roots along the same edge of $\Delta$, which are trivially bounded from below by 0. Without loss of generality consider the roots $a_{1,j-1}$ and $a_{1,j}$ for $1<j\le d_1$. We fix the $\pgl$ action on the space of roots by setting $a_{1,j-1}= -1$, $a_{2,1}= 0$ and $a_{m,d_m} = 1$. To show that $a_{1,j}-a_{1,j-1}$ attains its upper bound, we show that $a_{1,j}$ can not be arbitrarily close to $a_{2,1}= 0$. By assumption, the difference between the position of the tentacles
\begin{equation}
\rho_{1,j}-\rho_{1,j-1} = \sum\limits_{i=2}^n \sum\limits_{k=1}^{d_i} u_i\wedge u_1 \left(\log|a_{i,k}-a_{1,j}|-\log|a_{i,k}+1|\right)
\label{eq:tentacles}
\end{equation}
is bounded. We have that $a_{1,j} \in [-1,0)$ and $a_{i,k}\in [0,1]$ if $i\ge 2$. The function $\log|x-a_{1,j}|-\log|x+1|$ is increasing in $[0,1]$. On the other hand, 
\[
\sum\limits_{i=2}^n \sum\limits_{k=1}^{d_i} u_i\wedge u_1 = \left(\sum\limits_{i=1}^n d_iu_i\right)\wedge u_1 = 0
\]
by \Cref{eq:balancing}. Notice that for $i\in\{2,\dots,m\}$, $u_i\wedge u_1$ is positive for the first numbers and then negative for the rest (with a 0 in between if and only if $\Delta$ has another edge parallel to $\Gamma_1$). 

So, applying \Cref{lemma:proper} to the sequences $\log|a_{i,j}-a_{1,j}|-\log|a_{i,j}+1|$ and $u_i\wedge u_1$ (repeated $d_i$ times) we get that $\rho_{1,j}-\rho_{1,j-1}$ is non-negative. More importantly, since $u_2\wedge u_1 <0<u_n \wedge u_1$, we can subtract 
\begin{align*}
&\left(\log|a_{m,d_m}-a_{1,j}|-\log|a_{m,d_m}+1|\right)-\left(\log|a_{2,1}-a_{1,j}|-\log|a_{2,1}+1|\right)\\
= &\log|1-a_{1,j}|-\log|2|-\log|a_{1,j}|\\
\ge &-\log|2|-\log|a_{1,j}|
\end{align*}
from the right-hand-side in \Cref{eq:tentacles} and again by \Cref{lemma:proper} the result is still non-negative. This implies $\rho_{1,j}-\rho_{1,j-1}\ge -\log|2|-\log|a_{1,j}|$, which is arbitrarily large if $a_{1,j}$ is arbitrarily close to $0$. Since we assumed $\rho_{1,j}-\rho_{1,j-1}$ is bounded, $a_{1,j}$ is not arbitrarily close to $0$. 

Now we turn our attention to the $\RR^{m-3}$ component, assuming $m>3$. 
We choose a different representative of the $\pgl$-orbit on the roots by setting $a_{1,1}= 0$, $a_{3,1} = 1$ and $a_{m,d_m}= 2$. We will show that $a_{2,1}$ can not be arbitrarily close to $a_{1,1}= 0$. That implies that the paramater $a_{2,1}$ has a minimum in the preimage under $\tilde{\rho}$ of any compact set, and it attains it by continuity. Analogously, all other bounds regarding the $\RR^{m-3}$ component are achieved and we conclude the preimage is compact. 

To fix the $\RR^2$ action on the positions of the tentacles, we assume that the position of the $(1,1)$ tentacle and the $(m,d_m)$ tentacle are both 0. This is translating the amoeba by the vector
\[
w =  \dfrac{-\rho_{m,d_m}}{u_1\wedge u_m}u_1+\dfrac{-\rho_{1,1}}{u_m\wedge u_1}u_m.
\]
We show that if the position of the second tentacle after this translation,
\[
\hat{\rho}_{2,1} = \rho_{2,1} -\rho_{m,d_m}\dfrac{u_1\wedge u_2}{u_1\wedge u_m}-\rho_{1,1}\dfrac{u_m\wedge u_2}{u_m\wedge u_1},
\]
is bounded from below then $a_{2,1}$ is not arbitrarily close to $a_{1,1}= 0$. 

By \Cref{lemma:proper}, $\rho_{m,d_m}$ is non-positive, since $\log|x-2|$ is a decreasing function in $[0,2)$ and $u_m\wedge u_i$ is negative for the first values of $i$ and positive for the later. Since $u_1\wedge u_2<0< u_1\wedge u_m$, we have that $-\rho_n\dfrac{u_1\wedge u_2}{u_1\wedge u_m}$ is non-positive.

In both $\rho_{2,1}$ and $\rho_{1,1}$ there is a $\log|a_i|$ term which is arbitrarily large in absolute value if $a_{2,1}$ is close to $0$. As $a_{i,j}>1$ if $i\ge 3$, the only terms in $\rho_{1,1}$ and $\rho_{2,1}$ which can be arbitrarily large in absolute value are those corresponding to $a_{2,j}$ and $a_{1,j}$, respectively. 

In other words, the part that could grow arbitrarily large in absolute value in $\rho_{1,1}$ is
\[
\sum\limits_{j=1}^{d_2} u_2\wedge u_1 \log|a_{2,j}|
\]
and in $\rho_{2,1}$ it is 
\[
\sum\limits_{j=1}^{d_1} u_1\wedge u_2 \log|a_{2,1}-a_{1,j}|.
\]
Notice that 
\begin{equation}
|\log|a_{2,1}-a_{1,j}||  \ge |\log |a_{2,1}|| \ge |\log|a_{2,j}|.
\label{eq:logsproper}
\end{equation}
Let $c_1$ be the real number such that
\[
\sum\limits_{j=1}^{d_1} u_1\wedge u_2 \log|a_{2,1}-a_{1,j}| = c_1u_1\wedge u_2 \log|a_{2,1}|.
\]
By \Cref{eq:logsproper} we have that $c_1 \ge d_1$. Similarly, if $c_2$ is such that 
\[
\sum\limits_{j=1}^{d_2} u_2\wedge u_1 \log|a_{2,j}| = c_2u_2\wedge u_1 \log|a_{2,1}|
\]
then by \Cref{eq:logsproper} $c_2 \le d_2$. So the part of $\hat{\rho}_{2,1}$ which grows in absolute value is
\begin{align*}
&c_1u_1\wedge u_2 \log|a_{2,1}|-c_2u_2\wedge u_1 \log|a_{2,1}|\dfrac{u_m\wedge u_2}{u_m\wedge u_1}\\
&= \log|a_{2,1}|\dfrac{u_1\wedge u_2}{u_m\wedge u_1}\left(c_1u_m\wedge u_1 + c_2u_m\wedge u_2 \right)\\
&= \log|a_{2,1}|\dfrac{u_1\wedge u_2 \cdot u_m\wedge (c_1u_1+c_2u_2)}{u_m\wedge u_1} 
\end{align*}

Notice that $-d_1u_1-d_2u_2$ is the inner normal vector of the third side of the triangle formed by $\Gamma_1$ and $\Gamma_2$, so $u_m \in \cone(-d_1u_1-d_2u_2,u_1)$ since $m>3$. 
Thus, $u_m \wedge (d_1u_1+d_2u_2) > 0$. 
Because $c_1\ge d_1$ and $c_2 \le d_2$, we have that $c_1u_1+c_2u_2\in \cone(u_1,d_1u_1+d_2u_2)$. As $u_m \wedge u_1$ is also positive, we have that $u_m \wedge (c_1u_1+c_2u_2)> 0$. We conclude that
\[
\dfrac{u_1\wedge u_2 \cdot u_m \wedge (c_1u_1+c_2u_2)}{u_m\wedge u_1} > 0,
\]
which implies that $\hat{\rho}_{2,1}$ is negative and arbitrarily large in absolute value if $a_{2,1}$ is arbitrarily close to 0.
\end{proof}


\subsection{The Jacobian of $\rho$}
\label{subsec:jacobian}
In this subsection we proof that $\tilde{\rho}|_{\Hdelr}$ is a local diffeomorphism.

From now on, we sightly change the notation we have used so far to simplify the exposition and the computations. 
Instead of labelling the roots by pairs $(i,j)$ with $i\in [m]$ and $j\in [d_i]$, we relabel them as $a_1,\dots,a_n$ in the global cyclic order. Similarly, we relabel the $u_i$'s to agree with the labelling of the roots; that is, we have vectors $u_1,\dots,u_n$ where $u_i$ is the primitive inner normal vector of the edge of $\Delta$ that corresponds to the axis in which $\phi(a_i)$ vanishes. 
With this notation, we have that 
\[
\alpha(\phi)(t) = \prod_{i=1}^n (t-a_i)^{u_i}
\]
and that the $a_i$ and the $u_i$ correspond to a parametrization of a Harnack curve if
\begin{enumerate}
	\item $\sum\limits_{i=1}^n u_i = 0$.
	\item $a_1\le\dots \le a_n$. 
	\item The $u_1,\dots,u_n$ are ordered anticlockwise.
\end{enumerate}
We write $a= (a_1,\dots,a_n)$ and $U = (u_1,\dots,u_n)$ for short and we say $a$ and $U$ are \emph{cyclically ordered} if they satisfy the above conditions.
\medskip

Now we consider the Jacobian matrix $D$ of $\rho$ at a given point $a$. We have that 
\[
D_{i,j} =
\begin{cases*}
      \dfrac{u_i\wedge u_j}{a_i-a_j} & if $a_i \neq a_j$ \\
      0        & if $a_i=a_j$ but $i\neq j$\\
			-\sum\limits_{k\neq i}D_{i,k} & if $i= j$
\end{cases*} 
\]
In general, $D$ is a matrix that depends on $a$ and $U$, so we denote it as $D(a,U) = D(a_1,\dots,a_n;u_1,\dots,u_n)$. 

\begin{proposition}
\label{prop_local}
The map $\tilde\rho|_{\Hdelr}$ is a local diffeomorphism.
\end{proposition}
\begin{proof}
Let $a$ and $U$ be cyclically ordered. Let $T_a\pgl a$ be the tangent space of the orbit of $a$ under the $\pgl$ action at $a$ and similarly let $K$ be the kernel of the quotient $\RR^n\to \RR^n/\RR^2$ by the $\RR^2$ action defined in \Cref{eq:tentacles_action}. In other words,
\[
K =\{(r\wedge u_1,\dots,r\wedge u_n) \enspace \mid \enspace r\in \RR^2\}
\]
which is a linear space. Let us look at the relation of these spaces with $D$.

To compute the tangent space $T_a\pgl a$, recall that M\"obius transformations are of the form
\[
t \mapsto \dfrac{at+b}{ct+d}.
\]
We see that
\[
\left.\dfrac{\partial}{\partial \epsilon} t+\epsilon  \right|_{\epsilon = 0} = 1, \quad \left.\dfrac{\partial}{\partial \epsilon} (1+\epsilon)t \right|_{\epsilon = 0} = t,
\quad \left.\dfrac{\partial}{\partial \epsilon} \dfrac{t}{\epsilon t +1} \right|_{\epsilon = 0} = -t^2,
\]
so $T_a\pgl a$ is spanned by the vectors $(1,\dots,1)$, $(a_1,\dots,a_n)$ and $(a_1^2,\dots,a_n^2)$. Since
\[
\sum_{j=1}^n D_{i,j}= D_{i,i}-D_{i,i} = 0
\]
and
\[
\sum_{j=1}^n a_jD_{i,j} = \sum_{j=1}^n\frac{a_j \cdot u_i\wedge u_j}{a_i-a_j} - \sum_{j=1}^n\frac{a_i \cdot u_i\wedge u_j}{a_i-a_j} = -\sum_{j=1}^n u_i\wedge u_j = 0,
\]
we have that both $(1,\dots,1)$ and $(a_1,\dots,a_n)$ are in the kernel of $D$. In \Cref{rank_prop} we will proof that the kernel of $D$ is 2 dimensional, so the vector $(a_1^2,\dots,a_n^2)$ is not in the kernel. However, since $\rho$ descends to the map $\tilde\rho$, we have that $D \cdot (a_1^2,\dots,a_n^2)^\top \in K$. 

Since $D$ is symmetric, its image is the orthogonal complement of its kernel. So the image of $D$ is orthogonal to $(1,\dots,1)$ and $(a_1,\dots,a_n)$. 
On the other hand, $K$ is always orthogonal to $(1,\dots,1)$. By a similar argument as in the proof of \Cref{properpropo}, we have that $\sum\limits_{i=1}^n a_iu_i\ne 0$, so $(a_1,\dots,a_n)$ is never orthogonal to $K$. Thus, the intersection of the image of $D$ with $K$ is exactly one dimensional so it is spanned by $(a_1^2,\dots,a_n^2)$. 
This implies that the Jacobian of $\tilde\rho$ is injective, since no vector outside $T_a\pgl a$ vanishes under the composition of $D$ and the quotient $\RR^n \to \RR^n/\RR^2$.
\end{proof}

Now the only thing left to prove is the following:

\begin{proposition}
\label{rank_prop}
If $a$ and $U$ are cyclically ordered then the rank of $D(a,U)$ is $n-2$.
\end{proposition}
\begin{proof}
In \cite[Theorem 4]{kenyon2006planar} the authors prove this for the case where $\Delta$ is the dilation of a the unit triangle. They do it by showing that $D$ is a sum of $3\times 3$-block semidefinite positive matrices of rank $1$, each corresponding to the Jacobian of the unimodular triangle case. We here generalize this for any polygon $\Delta$. Let $e_1$, $e_2$, $e_3$ be the primitive normal vector of the standard unimodular triangle in clockwise order and let $T(a_i,a_j,a_k) = D(a_i,a_j,a_k,e_1,e_2,e_3)$ (see Equation (4.8) in \cite{kenyon2006planar}). 
We have that $T(a_i,a_j,a_k)$ is a rank 1 matrix with kernel generated by $(1,1,1)$ and $(a_i,a_j,a_k)$. We obtain thenon-zero eigenvalue by computing the image under $D$ of the cross-product of the two vectors in the kernel, $(1,1,1)\times (a_i,a_j,a_k)$. We get that the eigenvalue is
\[
\dfrac{(a_i-a_j)^2 +(a_j-a_k)^2+(a_k-a_i)^2}{(a_i-a_j)(a_j-a_k)(a_k-a_i)},
\]
which is always positive when $a_i<a_j<a_k$. 

Let $T_{i,j,k}(a_i,a_j,a_k)$ be the $n\times n$ matrix that restricts to $T(a_i,a_j,a_k)$ in the $3\times 3$ submatrix with indices $\{i,j,k\}$ and that is zero elsewhere.
We will show that if $a$ and $U$ are cyclically ordered, then $D(a,U)$ is a positive sum of matrices of the form $T_{i,j,k}(a_i,a_j,a_k)$. In other words we want to show that:
\[
D(a,U) \in \cone\left(\{ T_{i,j,k}(a_i,a_j,a_k) \enspace  \mid \enspace 1\le i<j<k \le n\}\right).
\]

To do so, we write each $u_i$ in the unique way $u_i= x_ie_1+y_ie_2+z_ie_3$ where $x_i,y_i,z_i\ge0$ and at most two are positive. With this notation we have that
\[
u_i\wedge u_j = x_iy_j+y_iz_j+z_ix_j -x_iz_j-y_ix_j-z_iy_j
\]
and that 
\[
\sum_{i=1}^n x_i =\sum_{i=1}^n y_i = \sum_{i=1}^n z_i.
\]
We call $c$ the constant above. For $i<j<k$ let
\[q_{i,j,k} = \det
\begin{pmatrix}
x_i & y_i &z_i\\
x_j & y_j & z_j\\
x_k & y_k & z_k
\end{pmatrix}.
\]

For all $i<j<k$ we have that $q_{i,j,k}\ge 0$. To see that, notice that the vectors $(x_i,y_i,z_i)$ are ordered cyclically along 
\[\RR^2\times\{0\} \cup \RR\times\{0\}\times\RR \cup \{0\}\times\RR^2,
\]
since they project to $U$ under 
$\left(\begin{smallmatrix}
1 & 0 & -1\\
0 & 1 & -1
  \end{smallmatrix}\right)$. Therefore, by the right-hand-rule, the determinant of that matrix is non-negative. Now we claim that
\[
\sum_{i<j<k}q_{i,j,k}T_{i,j,k}(a_i,a_j,a_k) = cD(a,U).
\]
To verify that claim, look at the coefficient of $\frac{1}{a_i-a_j}$ in the $(i,j)$-entry of the left hand side. We have that the coefficient is equal to
\begin{align*}
&x_iy_j\left(\sum_{k=1}^n z_k\right)+y_iz_j\left(\sum_{k=1}^n x_k\right)+z_ix_j\left(\sum_{k=1}^n y_k\right)\\
- &x_iz_j\left(\sum_{k=1}^n y_k\right)-y_ix_j\left(\sum_{k=1}^n z_k\right)-z_iy_j\left(\sum_{k=1}^n x_k\right)
&= c\cdot u_i\wedge u_j
\end{align*}
Therefore, $D(a,U)$ is the sum of positive semidefinite matrices, so its kernel is the intersection of the kernels of all of the summands. This implies that the kernel is the span of $(1,\dots,1)$ and $(a_1,\dots,a_n)$.
\end{proof}


%

\begin{proof}[Proof of Theorem \ref{theodiffeo}]
By \Cref{properpropo,prop_local}, $\tilde{\rho}$ is a proper local diffeomorphism whenever $a$ and $U$ are cyclically ordered. By \Cref{prop_hadamard}, $\tilde{\rho}|_{\Hdelr}$ is a diffeomorphism onto the space defined by \Cref{eq:codomain} modulo $\RR^2$.
\end{proof}

\section{From $\Hdelr$ to $\Hdel$}
The reason for the change of coordinates by $\tilde{\rho}$ is that fixing the position of the tentacles is fixing $C\setminus C^\circ$, which implies fixing $f|_{\partial \Delta}$ up to scaling by a constant. Polynomials using the remaining monomials $\inti(\Delta)\cap M$ were shown to be in correspondence with holomorphic differentials, in \cite[Section 2.2.4]{kenyon2006planar} for $\CP^2$ and in \cite[Lemma 4.3]{cretois2017vanishing} for smooth toric projective surfaces. The following lemma generalizes this to $X_\Delta$ for any $\Delta$. 
\begin{lemma}
\label{lemma_holomorphic}
Let $\Delta$ be any lattice polygon and let $C\subseteq X_\Delta$ be the vanishing set of a polynomial $f$ with Newton polygon $\Delta$ such that $\RR C$ is a Harnack curve. Then the space of holomorphic differentials on $C$ is isomorphic to the space of polynomials with coefficients in $\inti(\Delta)\cap M$, via the map
\begin{equation}
h(z_1,z_2) \mapsto \dfrac{h(z_1,z_2) dz_2}{\partial_{z_1} f(z_1,z_2)z_1z_2} 
\label{eq:differential}
\end{equation}
\end{lemma}
\begin{proof}
The map is injective and both spaces have dimension $g$, so it remains to prove that the differentials from \Cref{eq:differential} are holomorphic.
If $\Delta$ has a vertex at the origin with incident edges given by the coordinate axes, then the differentials from \Cref{eq:differential} are holomorphic over $\CC^2\cap C$ (see \cite[Lemma 4.3]{cretois2017vanishing}). 

Given any lattice-preserving affine map $A: M \otimes \RR \to M \otimes \RR$, that sending a polygon $\Delta'$ to $\Delta$, there is a dual map $A^\vee X_\Delta \to X_{\Delta'}$. In \cite{cretois2017vanishing} it is shown that the pullback of $A^\vee$ sends differentials of the form of \Cref{eq:differential} for $\Delta'$ to differentials of that form for $\Delta$. For each vertex $v$ of $\Delta$, consider the lattice preserving affine map that sends the positive orthant to the cone spanned by $\Delta$ from $v$. Then the differentials from \Cref{eq:differential} are holomorphic in the intersection of $C$ with the affine chart corresponding to $v$. Doing that for every vertex we get that they are holomorphic in all of $C$.   

If $\Delta$ is not a smooth polygon, then such a lattice-preserving map does not exist. However, given any vertex $v$ of $\Delta$, there is a lattice-preserving map that sends the cone spanned by $(1,0)$ and $(p,q)$ to the cone spanned by $\Delta$ from $v$, for some suitable $p,q \in \NN$. Let $\Delta'$ be the preimage of $\Delta$ under such map. By the same arguments as in \cite[Lemma 4.3]{cretois2017vanishing}, the differentials of the form of \Cref{eq:differential} for $\Delta'$ are holomorphic in $(\CC\times \CC^*)\cap C$. This implies that the pullback is holomorphic in $((\CC^*)^2\cup L) \cap C$, where $L$ is the axis that corresponds to the edge contained in the image under $A$ of the coordinate axis $\{x_1 = 0\}$. This can be done for every edge of $\Delta$. Since $C$ does not contain the intersection any two axes, it follows that the differentials in \Cref{eq:differential} are holomorphic over all $C$. 
\end{proof}

\begin{proposition}
\label{boundarypropo} The areas of the holes of the amoeba are global coordinates for the moduli space of Harnack curves with fixed Newton polygon $\Delta$ and fixed boundary points. Moreover, the moduli space of Harnack curves with fixed boundary is diffeomorphic to $\RR_{\geq0}^g$.
\end{proposition}
\begin{proof}
Recall the diagram from \Cref{eq:fixed_boundary} in the introduction. The first map sends a Harnack curve with fixed boundary (that is, we fix $f|_{\partial \Delta}$) to the bounded Ronkin intercepts. By \Cref{lemma_holomorphic}, the differential of that map is the period matrix of $C$ (see \cite[Proposition 6]{kenyon2006planar} and \cite[Theorem 3]{cretois2017vanishing}). The second map, from the bounded intercepts to the areas of the holes in the amoeba, is also a local diffeomorphism because its differential is diagonally dominant (see \cite[Proposition 10]{kenyon2006planar}). The areas of the holes of the amoeba are non-negative and the composition of the two maps is proper over $\RR_{\ge0}^g$ (see \cite[Theorem 6]{kenyon2006planar}). All of these facts are proven in \cite{kenyon2006planar} and none of the arguments used there require that $X_\Delta = \CP^2$. Again, by \Cref{prop_hadamard} the composition of the maps is a diffeomorphism onto $\RR_{\geq0}^g$.
\end{proof}

Notice that the positions of the tentacles are also well-defined numbers for non-rational Harnack curves: they are simply the evaluation of $J u_i$ on the points $C\cap L_i$. So by \Cref{theodiffeo} and \Cref{rationalmoduli,boundarypropo} we have that the positions of the tentacles of the amoeba modulo translation together with the areas of the holes of the amoeba are global coordinates for $\Hdel$. Hence, we have proved that.

\begin{theorem}
\label{modulitheorem}
Let $\Delta$ be a lattice polygon with $m$ sides, $g$ interior lattice points and $n$ boundary lattice points. Then $\Hdel$ is diffeomorphic to $\RR^{m-3}\times\RR_{\geq0}^{n+g-m}$.
\end{theorem}


\section{The Compactification of $\Hdel$}
The goal of this section is to construct a natural compactification $\Hdelc$ of $\Hdel$ by collections of `patchworkable' Harnack curves.
\subsection{Abstract tropical curves}
We begin with a review of abstract tropical curves and of $\Mtrop$, the moduli space of tropical curves with $n$ legs and genus $g$. For details of this construction see \cite{caporaso2011algebraic}. 
\medskip

A \emph{weighted graph with $n$ legs} $\cG$ is a triple $(V,E,L,w)$ where 
\begin{itemize}
	\item $(V,E)$ is a perhaps non-simple connected graph, that is, we allow multiple edges and loops.
	\item $L: [n] \rightarrow V$ is a function which we think of as attaching $n$ labelled \emph{legs} at vertices of the graph. 
	\item $w$ is a function $V\rightarrow \NN$ which we call the \emph{weights} of the vertices.
\end{itemize}
The \emph{genus} of $\cG$ is the usual genus of $(V,E)$ plus the sum of the weights on all vertices; that is 
\[\text{genus}(G) = \sum_{v \in V} w(v) -|V|+|E|+1.\] 

An \emph{isomorphism} between two graphs $\cG_1 = (V_1,E_1,L_1,w_1)$ and $\cG_2 = (V_2,E_2,L_2,w_2)$ is a pair of bijections $\phi_V: V_1\rightarrow V_2$ and $\phi_E: E_1\rightarrow E_2$ such that 
\begin{itemize}
	\item For any edge $e\in E_1$ and any vertex $v\in V_1$, $\phi_E(e)$ is incident to $\phi_V(v)$ if and only if $e$ is incident to $v$.
	\item $L_2 = \phi_V(L_1)$.
	\item $w_1(v) = w_2(\phi_V)$.
\end{itemize}

Let $\cG/e$ denote the usual contraction of $\cG$ over an edge $e$ with the following change of weights: if we contract a non loop $ab$, then the contracted vertex gets weight $w(a)+w(b)$. If the contracted edge is a loop on $a$, then the weight of $a$ is increased by 1. Observe that the genus is invariant under contraction.

We say that a weighted graph $\cG$ is \emph{stable} if every vertex with weight $0$ has degree at least $3$ and every vertex with weight $1$ has positive degree. An \emph{(abstract) tropical curve} is a pair $(\cG,l)$ where $\cG$ is a stable weighted graph and $l$ is a function that assigns \emph{lengths} to the edges of $\cG$, in other words, $l$ is a function $l : E(\cG) \rightarrow \RR_{\geq 0}^{|E(\cG)|}$. The \emph{genus} of the tropical curve is the genus of $\cG$. An isomorphism between two abstract tropical curves $(\cG_1,l_1)$ and $(\cG_2,l_2)$ is an isomorphism $\phi$ of the weighted graphs $\cG_1$ and $\cG_2$ such that $l_1 = l_2 \circ \phi_E$, or such that one is the result of contracting an edge of length $0$ form the other, or the transitive closure of these two relations.

Given a weighted stable graph $\cG$, one can identify the space of all tropical curves over $\cG$ with $\RR_{\geq 0}^{|E(\cG)|}$. We define $\Mtrop(\cG) :=  \RR_{\geq 0}^{|E(\cG)|}$ and
\[
\Mtrop := \left.\left(\bigsqcup\limits_{\cG \text{ stable}} \Mtrop(\cG)  \right) \middle/ \sim \right. ,
\]
where $\sim$ denotes isomorphism. This is a connected Hausdorff topological space which parametrizes bijectively isomorphism classes of tropical curves. It is covered by $\Mtrop(\cG)$ where $\cG$ runs over all 3 valent graphs with all vertices of weight 0. For such graphs we have that $\Mtrop(\cG)$ is just $\RR^{3g+n-3}$. However $\Mtrop$ is not a manifold as there are triples of graphs of this form glued along codimension 1. 

To compactify this space we allow lengths to be infinite. Let $\RR_\infty = \RR_{\geq 0} \sqcup \{\infty\}$ be the one point compactification of $\RR_{\geq 0}$. An \emph{extended} tropical curve $(\cG,l)$ consists of a stable weighted graph $\cG$ and a length function $l : E(\cG) \rightarrow \RR_\infty^{|E(\cG)|}$. We define isomorphism classes of extended tropical curves in the same way as for tropical curves. This way we define $\overline{\Mtrop(\cG)} :=  \RR_\infty^{|E(\cG)|}$ and 
\[
\overline{\Mtrop} := \left. \left(\bigsqcup\limits_{\cG \text{ stable}} \overline{\Mtrop(\cG)}  \right) \middle/ \sim \right.
\]
This is a compact hausdorff space with $\Mtrop$ as an open dense subspace.

\subsection{Ronkin intercepts}
\label{sec:compact} 

From now on, lightly abusing notation, we call elements of $\Hdel$ curves and denote one of them $C$, even though by definition they are equivalence classes of Harnack curves. We say that a polynomial vanishes on $C$ if its zero locus is in the equivalence class $C$.

The expanded spines of different Harnack curves in the same equivalence class in $\Hdel$ only differ by translations. In particular, the combinatorial type and the lengths of the bounded edges remain the same. So given a curve $C\in\Hdel$, we have a well defined abstract tropical curve structure for its expanded spine $\spine(C)\in\Mtrop$: fix a labelling of the boundary segments of $\Delta$ by $[n]$ in a cyclical way and let $L(\spine(C))(k)$ be the vertex incident to the ray corresponding to the segment labelled $k$. This induces a map

\[
\spine : \Hdel \rightarrow \Mtrop.
\]

Recall that Ronkin intercepts are the coefficients of the tropical polynomial defining the expanded spine.

\begin{proposition}
\label{prop_ronkin}
The Ronkin intercepts modulo translations of the graph of the Ronkin function can be taken as global coordinates for $\Hdel$.
\end{proposition}
Notice that translations of the Ronkin function are the same as translations of the amoeba.
\begin{proof}
Since we proved in \Cref{boundarypropo} that composition of the maps in \Cref{eq:fixed_boundary} is a diffeomorphism, each of the maps themselves are diffeomorphisms. This implies that the bounded Ronkin intercepts can be taken as global coordinates for Harnack curves with fixed boundary.

Now, \Cref{theodiffeo} says that the positions of the amoeba tentacles can be taken as global coordinates for rational Harnack curves, and it is easy to recover the unbounded Ronkin intercepts from the positions of the tentacles as follows.
Let $\rho_i$ be the position of a tentacle. It corresponds to a segment in $\partial \Delta$ lying between two lattice points. Let $c_i$ and $c_{i+1}$ be the intercepts corresponding to those points. It is straightfoward that $\rho_i = c_i-c_{i+1}$. This implies that the map
\[
\begin{Bmatrix}
\text{Positions of}\\
\text{amoeba}\\
\text{tentacles}
\end{Bmatrix} /\RR^2
\rightarrow
\begin{Bmatrix}
\text{Unbounded}\\
\text{Ronkin}\\
\text{intercepts}/\RR^3
\end{Bmatrix}
\]
is a linear bijection when restricted to $\left\{\sum\limits_{i = 1}^n\rho_i= 0\right\}$. 
\end{proof}
\begin{proposition}
\label{prop:piecewise}
The map $\Hdel \to \Mtrop$ is a piecewise linear topological embedding.
\end{proposition}
\begin{proof}
By \Cref{prop_ronkin}, we can take the Ronkin intercepts as global coordinates of $\Hdel$. The Ronkin intercepts can be recovered from $\spine(C)$ and $\Delta$ up to translations of the amoeba. Computing the lengths of the bounded edges of a tropical curve from the tropical polynomial is again solving a system of linear equations. So, over every component $\Mtrop(\mathcal(G))$ of the co-domain, the map $\spine$ is linear.
\end{proof}




\subsection{Harnack meshes}
%


\begin{definition}
Let $B\subseteq M$ be a finite set of affine dimension 2. We define $\mathcal{H}_B$ the subset of $\mathcal{H}_{\Conv(B)}$ consisting of curves $C$ such that for every $v\in \Conv(B)_M$ the corresponding component $E_v$ in $\RR^2\setminus \AAA(C)$ exists if and only if $v\in B$.
\end{definition}
This is well defined since the existence of $E_v$ depends only on the equivalence class of a Harnack curve. 
By \Cref{modulitheorem}, $\mathcal{H}_B$ is diffeomorphic to $\RR^{|B|-3}$. 

\begin{definition}
\label{defn:mesh}
Consider a regular subdivision $\SSS$ of $\Delta$ with facets $\{B_1,\dots,B_s\}$ and let $\Delta_i = \Conv(B_i)$. A \emph{Harnack mesh over $\SSS$} is a collection of curves $(C_1,\dots,C_s)$ with $C_i\in \mathcal{H}_{B_i}$ such that there exists a polynomial $f$ with $f|_{\Delta_i}$ vanishing on $C_i$. 
We denote
\[\Hdel(\SSS) \subseteq \prod\limits_{i=1}^s  \mathcal{H}_{B_i}\]
for the space of all Harnack meshes over $\SSS$.
\end{definition}

Notice that $\Hdel$ is equal to the disjoint union of all $\Hdel(\SSS)$ where $\SSS$ has exactly one face, i.e., all $\SSS$ of the form $\{B\}$ with $\Delta = \Conv(B)$. 

The existence of such $f$ in the definition above is equivalent to the $C_i$ agreeing on their common boundary. That is, given $\Delta_i$ and $\Delta_j$ such that $\Gamma = \Delta_i \cap \Delta_j$, the distances between the tentacles of $C_i$ corresponding to $\Gamma$ are the same to the distances between the respective tentacles in $C_j$. 

Given any Harnack mesh $(C_1,\dots C_s) \in \Hdel(\SSS)$ we can define its expanded spine as an extended tropical curve. Let $\spine_i$ be the expanded spines of $C_i$. For each edge $\Delta_i\cap \Delta_j$, glue the expanded spines $\spine_i$ and $\spine_j$ by removing the legs corresponding to that edge and placing instead an edge of infinite length for each primitive segment in $\Delta_i\cap \Delta_j$ between the two vertices that were incident to the corresponding leg. The remaining legs are labelled by the boundary segments of $\Delta$. This way we have a map

\[
\spine_\SSS: \Hdel(\SSS) \rightarrow \overline{\Mtrop},
\]
which is an embedding, by \Cref{prop_ronkin}.

\begin{definition}
Let $\Delta$ be a lattice polygon. The \emph{compactified moduli space of Harnack curves} is
\[
\overline{\Hdel} := \bigsqcup_\SSS \Hdel(\SSS),
\]
where the union runs over all regular subdivisions $\SSS$ of $\Delta_M$. We give it the coarsest topology that makes the map
\[
\spine: \overline{\Hdel} \to  \overline{\Mtrop}
\]
defined by $\spine|_{\Hdel(\SSS)} := \spine_\SSS$, continuous.
\end{definition}
We will prove that this in fact is a compactification of $\Hdel$, i.e. a compact space where $\Hdel$ is dense.
\medskip

Harnack meshes are essentially collections of Harnack curves that can be patchworked into another Harnack curve, except that we allow singularities and non transversal intersection with the axes. We call a subdivision $\SSS$ of $\Delta$ is \emph{full} if it uses all the points, that is $\bigcup\limits_{B_i\in\SSS} B_i = \Delta_M$. So we can only do true patchworking with them whenever $\SSS$ full. 



\begin{proposition}
\label{harnack_patchwork}
Let $\SSS$ be a full regular subdivision of $\Delta$ with lifting function $h$ and let $(C_1,\dots,C_s)\in \Hdel(\SSS)$ be a Harnack mesh. Then there exist polynomials $f_1,\dots,f_s$ vanishing on $C_1,\dots,C_s$ such that the polynomial $f_t$ obtained by patchworking them vanishes on a Harnack curve.
\end{proposition}
\begin{proof}
$\SSS$ being full implies that each $C_i$ is smooth and non transversal in the boundary, so that they can be patchworked together.
The topological type of $C_i$ can be obtained by patchworking using a unimodular triangulation of $\Delta_i$ and the `Harnack' sign pattern $v\mapsto (-1)^{v_1v_2}$. By taking a unimodular triangulation of $\Delta$ that refines $\SSS$ and the Harnack sign pattern, the result of patchworking is a Harnack curve. It must have the same topological type as the patchwork using $\SSS$ and the curves $C_i$ when the polynomials $f_i$ are chosen with the same sign as the polynomial that results from patchworking $\Delta_i$ with the Harnack sign configuration.
\end{proof}

\begin{proposition}
\label{cell_dim}
Let $\SSS$ be a regular subdivision of $\Delta$. Then $\Hdel(\SSS)$ is diffeomorphic to $\RR^{n+g -\dim(\sigma(\SSS))}$ where $\sigma(\SSS)$ is the cone in the secondary fan of $\Delta$ corresponding to $\SSS$.
\end{proposition}
\begin{proof}

First consider the case where $\SSS$ has a single facet $B$. By \Cref{modulitheorem}, $\Hdel(\SSS) = \mathcal{H}_B \cong \RR^{|B|-3}$ and the result follows from $\dim(\sigma(\SSS)) \allowbreak = n+g+3-|B|$.

Now let $\SSS$ be any full regular subdivision with facets $B_1,\dots,B_s$ and let $\Delta_i = \Conv(B_i)$. Let us compute $\dim(\sigma(\SSS))$. Note that for $h\in \sigma(\SSS)$, fixing $h$ for 3 affinely independent points of $B_i$ fixes $h$ on all of $B_i$. 
First, $\Delta_1$ is fixed after fixing 3 affinely independent points. Suppose $\Delta_2$ is adjacent to $\Delta_1$ and let $v_1$ be any element of $B_2\setminus B_1$. We have that $h(v_1)$ can take any positive real value. However, after fixing $h(v_1)$, all of $h|_{\Delta_2}$ is determined. Furthermore, if $B_i$ shares sides with both $B_1$ and $B_2$ then $f_i$ is also determined.

 Let $I_1 \subseteq [s]$ is the minimum set containing 1 and 2 and such that 
$\bigcup\limits_{i \in I_1} \Delta_i$
is a convex set. Then if $h$ is determined for $B_1$ and $v_1$, it is determined for all of $\bigcup\limits_{i \in I_1} B_i $. If $I_1 \neq [s]$, we can repeat last step; choose a vertex $v_2$ in a facet $F_i$ with $i\neq I_1$ but adjacent to a facet $F_j$ with $j\in I_1$. Further determining the value of $h$ on $v_2$ determines $h$ on a set of facets indexed by $I_2$. We can repeat this until $I_d = [s]$. By construction, $d+3 = \dim(\sigma(\SSS))$. We can reorder $[s]$ so that without loss of generality we can assume that $I_i =[k_i]$ for some $1\leq k_i \leq s$. If $\SSS$ is not full, then 
\[
\dim(\sigma(\SSS)) =d+3+|\Delta_M \backslash \bigcup_{i=1}^s B_i|,
\]
since the value of $h$ for points in $\Delta_M \backslash \bigcup\limits_{i=1}^s B_i$ can be arbitrary as long as it is large enough.

Now let us compute $\Hdel(\SSS)$. Recall that the global coordinates for $\Hdel$ that were taken in \Cref{modulitheorem} consist of areas of ovals, distances between consecutive parallel tentacles, and the position of the first tentacle of every edge after the second. 
We start with $B_1$ and we have that $\mathcal{H}_{B_1} \cong \RR^{|B_1|-3}$. The distances between the tentacles corresponding to the edge $\Gamma = \Delta_1\cap\Delta_2$ are the same for $C_1$ and $C_2$. So, these $|\Gamma|-2$ parameters of $\mathcal{H}_{B_2}$ are fixed if we fix $C_1$. The subspace of $\mathcal{H}_{b_2}$ of curves $C_2$ that agree with $C_1$ on the boundary is isomorphic to $\RR^{|B_2|-|\Gamma|-1}$. Let 
\[
q_i = \left|B_i \cap\left(\bigcup\limits_{j=1}^i B_i\right)\right|.
\]
Similarly, we have that the subspace of Harnack curves $C_i \in \mathcal{H}_{B_i}$ compatible with $C_1,\dots C_{i-1}$ is isomorphic to $\RR^{|B_i|-q_i-1}$ if $\Delta_i$ only shares an edge with one of $\Delta_1,\dots,\Delta_{i-1}$. If it shares edges with 2 of $\Delta_1,\dots,\Delta_{i-1}$, then there are $q_i-3$ distances between tentacles being fixed, so the subspace of Harnack curves $C_i \in \mathcal{H}_{B_i}$ compatible with $C_1,\dots C_{i-1}$ is isomorphic to $\RR^{|B_i|-|\Gamma|}$. The last statement also holds when $\Delta_i$ shares edges with 3 or more of $\Delta_1,\dots,\Delta_{i-1}$, since the position of the first tentacle of each edge after the second gets fixed. By construction, $B_i$ shares only one edge with the previous polygons if and only if $i =2$ or $i =k_j+1$, for $1\le j< d$. We have that 
\begin{align*}
\dim(\Hdel(\SSS)) &= \sum\limits_{i=1}^s |B_i|-3-d \\
&=  |\bigcup \SSS|-3-d	\\
&=  n+g - \dim(\sigma(\SSS)).
\end{align*}
\end{proof}

\begin{proposition}
\label{closure_prop}
Let $\Delta$ be a lattice polygon and $\SSS$ be a regular subdivision of $\Delta_M$. Then $\spine_\SSS(\Hdel(\SSS)) \subseteq \overline{\spine(\Hdel)}$, where $\overline{\spine(\Hdel)}$ is the closure of $\spine(\Hdel)$ in $\overline{\Mtrop}$.
\end{proposition}
\begin{proof}
Suppose $\SSS$ is a full subdivision of $\Delta$. Choose a height function $h\in\sigma(\SSS)$. By \Cref{harnack_patchwork}, for any Harnack mesh $\Mesh = (C_1,\dots,C_s)$ in $\Hdel(\SSS)$ there exists polynomials $f_1,\dots,f_s$ and $t_0>0$ such that for any $0<t<t_0$ the curve $C_t$ obtained by patchworking is in $\Hdel$. So, we have a path $(0,t_0)\rightarrow\Hdel$. For each facet $B_i$ of $\SSS$ there is a family of polynomials $\{f_t^i \enen t \in (0,t_0)\}$ with real coefficients such that $f_t^i$ vanishes on $C_t$ and every coefficient of $f_t^i$ outside $\Conv(B_i)$ goes to 0 as $t$ goes to 0. This follows from picking the height function $h^i$ affinely equivalent to $h$ such that points in $B_i$ have height 0 and doing patchworking with $h^i$.
The limit $\lim\limits_{t\to \infty} f_t^i = f_i$ vanishes on $C_i$. 

As $t$ goes to 0, the lengths of the edges of $\spine(C_t)$ that corresponds to the interior of $B_i$ tend to the lengths of the edges of $\spine(C_i)$. Doing this for every $i$ we have that all the finite lengths of $\spine_\SSS(\Mesh)$ agree with the lengths of $\lim \limits_{t\rightarrow 0}\spine(C_t)$. The edges going to infinity are precisely those dual to primitive segments of $\SSS$. Then, $\spine(C_t)$ forms a path $(0,t_0)\rightarrow \spine(\Hdel)$ such that the limit of this path when $t$ goes to 0 is $\spine(\Mesh)$. So, $\Mesh \in \overline{\Hdel}$.

If $\SSS$ is not full, let $\SSS'$ be the subdivision whose facets are $B' = \Conv(B)_M$ for each facet $B\in \SSS$. That is, $\SSS'$ is the finest full subdivision that coarsens $\SSS$. It is regular, as we can take a height function a height function in $\sigma(\SSS)$ and linearly extrapolate in each $\Delta_i$ to make it full. As the expanded spine is continuous, even when ovals contract, we have that 
\[
\spine_{\SSS'}(\Hdel(\SSS'))\subseteq \overline{\spine_\SSS(\Hdel(\SSS))} \subseteq \overline{\Hdel}.
\] 
\end{proof}


\begin{lemma}
\label{sobrey_lemma}
Let $\Delta$ be a lattice polygon. Then $\overline{\spine(\Hdel)} = \spine(\Hdelc)$.
\end{lemma}
\begin{proof}
Proposition \ref{closure_prop} implies that $\spine(\Hdelc) \subset \overline{\spine(\Hdel)}$.

For the other containment, let $C_1,C_2,\dots$ be a sequence of curves in $\Hdel$ such that their expanded spines converge to a point in $G\in\overline\Mtrop$. We call connected components of $G$ the components obtained by deleting from $G$ all edges of infinite length. Notice that vertices of an expanded spine correspond to polygons inside $\Delta$ given by the regular subdivision dual to the expanded spine. This association is carried on to the limit, so the connected components $G_1,\dots, G_s$ induce a regular subdivision $\SSS=\{B_1,\dots, B_s\}$ of $\Delta$. For a connected component $G_i$ of $G$, we can choose polynomials $f_1^i,f_2^i,\dots$ vanishing on $C_1,C_2\dots$ such that they converge to a polynomial $f^i$ which vanishes on a curve whose expanded spine is $G_i$. This can be done, for example, by picking a vertex of $G_i$ and fixing it to be in the origin, i.e., translating the amoebas of $C_1,C_2\dots$ so that the corresponding vertex in the expanded spine is always at the origin. Since the limit of Harnack curves is a Harnack curve (see \cite[Remark 2]{mikhalkin2001amoebas}), $f^i$ vanishes on a Harnack curve $C^i \in \mathcal{H}_{B_i}$. The collection $\Mesh = (C_1,\dots,C_s)$ is a Harnack mesh over $\SSS$ and we have that $\spine(\Mesh) = G$.
\end{proof}

\begin{corollary}
\label{coro_poset}
Let $\Delta$ be a lattice polygon. Then,
\[\overline{\Hdel(\SSS)} = \bigcup \limits_{\mathcal{T} \leq \SSS} \Hdel(\mathcal{T}),\]
where the union runs over all subdivisions $\mathcal{T}$ of $\Delta$ that refine $\SSS$.
\end{corollary}

\begin{theorem}
\label{thm_compact}
Let $\Delta$ be a lattice polygon. The stratification of $\Hdelc$ by $\Hdel(\SSS)$ is a cell complex with a poset isomorphic to the face poset of the secondary polytope $\Sec(\Delta_M)$ given by its faces. 
\end{theorem}
\begin{proof}
The faces of $\Sec(\Delta_M)$ are in correspondance with regular subdivisions. By \Cref{cell_dim}, $\Hdel(\SSS)$ has the same dimension as the face of $\Sec(\Delta_M)$ corresponding to $\SSS$. By \Cref{coro_poset}, the boundary of $\overline{\Hdel(\SSS)}$ consists of $\Hdel(\mathcal{T})$ for every subdivision $\mathcal{T}$ that refines $\SSS$. Similarly, the faces contained in the face of $\Sec(\Delta_M)$ corresponding to $\SSS$ are those corresponding to refinements of $\SSS$.
\end{proof}

\begin{example}
\label{ex:hdelc}
\begin{figure}[h]
	\centering
		\includegraphics[width=0.75\textwidth]{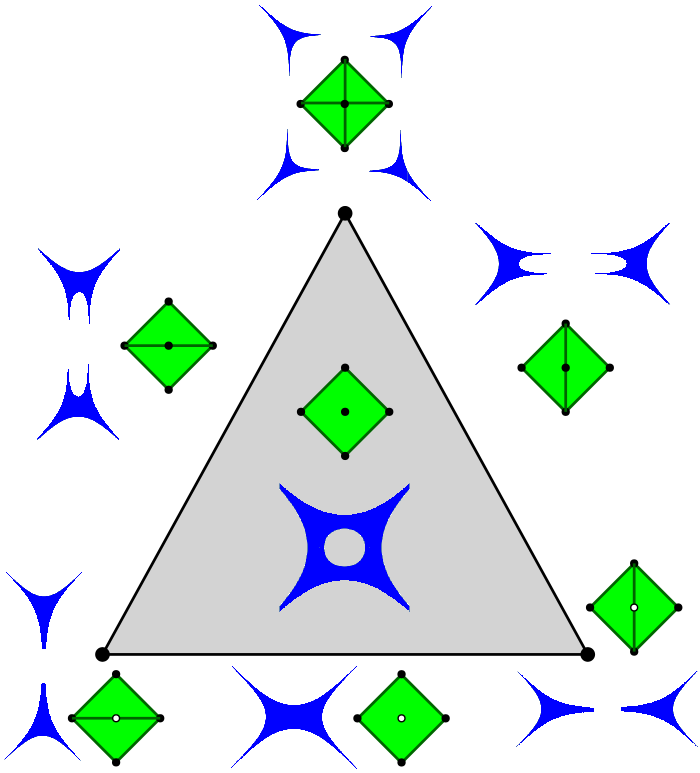}
	\label{fig:ejecomp}
	\caption{$\Hdelc$ for $\Delta = \Conv((1,0),(0,1),(-1,0),(0,-1))$}.
	\label{fig_hdelc}
\end{figure}

Let $\Delta := \Conv((1,0),(0,1),(-1,0),(0,-1))$. We have that $\Sec(\Delta)$ is a triangle. Figure \ref{fig_hdelc} shows the space $\Hdelc$ together with the subdivisions of the corresponding face in $\Sec(\Delta)$ and the amoebas of the corresponding Harnack meshes. The horizontal coordinate represents the relative position of the tentacles. This is parametrized, for example, by $\rho_1+\rho_3$. Going to the left stretches the amoeba vertically while going to the right stretches it horizontally. The vertical coordinate corresponds to the area of the oval, where going downwards decreases the area while going upwards increases it. The bottom open segment corresponds to $\Hdelr$, and that segment together with the interior face corresponds to $\Hdel \cong \RR\times \RR_{\geq0}$. 
\end{example}

\begin{remark}
The regular subdivision corresponding to the open cell on which a Harnack mesh $\Mesh$ sits in is not the same as the subdivision dual to the graph of $\spine(\Mesh)$.
In fact, the two subdivisions are not always comparable with the refinement order. 
For example, the subdivision corresponding to the bottom edge of \Cref{fig_hdelc} is the square without the middle point, but the subdivision dual to the graph of a spine in that cell would be the subdivision with 4 triangles.
\end{remark}

\section{Questions and future directions}
\subsection{$\Hdelc$ as a CW-complex}
\label{subsec:CW}
We begin by suggesting the following strengthening of \Cref{thm_compact}:

\begin{conjecture}
\label{conj:cw}
The compactified moduli space $\Hdelc$ is a CW-complex. 
\end{conjecture}

If \Cref{conj:cw} is true, then $\Hdelc$ is automatically regular by construction. Since regular CW-complexes with the same poset are isomorphic, by \Cref{thm_compact}, $\Hdelc$ would be isomorphic to $\Sec(\Delta_M)$. To show that $\Hdelc$ is CW is essentially showing that for any regular subdivision $\SSS$ of $\Delta$, $\spine_\SSS(\Hdel(\SSS))$ is a closed ball. 
It is enough to prove the following:

%

\begin{conjecture}
\label{conj:local}
Let $\Mesh\in \Hdel(\SSS)$ be a Harnack mesh and $\SSS'$ be a coarsening of $\SSS$, then there is a neighborhood of $\Mesh$ in $\overline{\Hdel(\SSS')}$ homeomorphic to a half space of dimension $\dim(\Hdel(\SSS'))$.
\end{conjecture}

Since the poset of $\Hdelc$ is Eulerian by \Cref{thm_compact}, \Cref{conj:local} implies that the closure of the cells of $\Hdelc$ are closed balls by (a reformulation of) Pointcar\'e's conjecture. This argument was recently used by Galashin Lam and Karp in order to prove that the positroid stratification of the totally non negative Grassmannian is a CW-complex \cite{GKL}. It is worth remarking that Harnack curves enjoy several similarities with the total positivity phenomenon (see, for example, \cite[Section 5.2]{kenyon2006dimers} or the proof \Cref{rank_prop}). In the next subsection we will see that \Cref{conj:local} holds when $\SSS$ is full.


\subsection{$\Hdelc$ as a manifold with generalized corners.}
The above discussion suggest to study topological charts of $\Hdel$. We can be more ambitious and try to endow $\Hdel$ with a smooth structure.
\Cref{modulitheorem} is already a description of $\Hdel$ as a smooth manifold with corners. 
A natural question is whether we can extend this smooth structure to $\Hdelc$.
A desirable trait of such a smooth structure (besides being compatible with the chart given by \Cref{modulitheorem}) is that the cell complex structure from \Cref{thm_compact} can be recovered from it.
However, secondary polytopes are not always simple polytopes 
and manifolds with corners lack the capacity to describe non-simple vertices. 
To mend this, we turn our attention to a wider category, namely that of manifolds with \emph{generalized corners}, or \emph{$gc$-manifolds}, as defined in \cite{gcmanifolds}.

\begin{definition}[\cite{gcmanifolds}]
A \emph{$g$-chart} of a topological space $X$ is a triple $(\phi, \LL, U)$ such that:
\begin{itemize}
	\item $\LL$ is a weakly toric monoid, i.e. a semi-lattice of the form $\LL = \ZZ^s\cap \sigma$ where $s$ is a positive integer and $\sigma\subseteq \RR^s$ is a rational polyhedral cone.
	\item $U$ is an open subset of $\homm(\LL, \RR_{\ge 0})$, i.e. the space of monoid morphisms from $\LL$ to the monoid $(\RR_{\ge 0},\cdot)$ with the weakest topology that makes evaluation on a point $q\in \LL$ continuous.
	\item $\phi:U \to X$ is a topological embedding to an open subset $\phi(U)\subseteq X$.
\end{itemize}
\end{definition}

We call $X$ a $gc$-manifold if it has a \emph{$g$-atlas}, that is, a collection of $g$-charts covering $X$ satisfying certain compatibility conditions on the transition functions. These conditions depend on the monoids, but we restrain from explaining them in detail in this paper to avoid overextension.
Easy examples of $gc$-manifolds are $\homm(\NN, \RR_{\ge 0}) \cong \RR_{\ge 0}$ and $\homm(\ZZ, \RR_{\ge 0}) \cong \RR$. 
So as a $gc$-manifold, $\Hdel\cong \homm(\NN^{m-3}\times\ZZ^{n+g-m}, \RR_{\ge 0})$ by \Cref{modulitheorem}.

\begin{proposition}
Let $\Mesh\in \Hdel(\SSS)$ be a Harnack mesh where $\SSS$ is a full subdivision. Then there exists a $g$-chart around $\Mesh$. Moreover, all charts provided this way are compatible with each other.
\end{proposition}

\begin{proof}
Since $(\RR\cup\{\infty\}, +)$ is isomorphic as a monoid to $(\RR_{\ge 0}, \cdot)$ with $x\mapsto e^{-x}$ as isomorphism, $g$-charts can be equivalently defined to be homeomorphisms from open subsets of `affine tropical toric varieties', i.e. from open subsets $U\subseteq \homm(\LL, \RR\cup\{\infty\})$.
Consider a graph $G$ embedded in $\RR^2$ that is dual to the subdivision $\SSS$. 
In particular, the edges of $G$ have a prescribed slope.
The lengths of these edges satisfy linear equations with integer coefficients given by the circuits of $G$ (two for each circuit). 
These equations are binomial relations under the isomorphism $(\RR\cup\{\infty\},+) \cong (\RR_{\ge0}, \cdot)$.
Thus, the edges of $G$ (which correspond to edges of $\SSS$ in the interior of $\Delta$) generate a toric monoid $\LL_\SSS$ under these relations.

Let $\Mesh'$ be a Harnack mesh close enough to $\Mesh$. The spine $\spine(\Mesh')$ has a subgraph $G_i$ which is very close to $\spine(C_i)$ for each curve $C_i\in \Mesh$. These subgraphs are glued together with edges of very large (possibly infinite) length. Contracting these subgraphs results in the graph dual to $\SSS$, so the distances between these graphs induce a homomorphism $\phi_{\Mesh'}: \LL_\SSS \to \RR\cup\{\infty\}$. 

The coordinates of a Harnack mesh in $\Hdel(\SSS)$ encode the same information as the spines of each curve in the Harnack mesh.
Since $\Mesh'$ is close enough to $\Mesh$ there exists a mesh $\Hdel(\SSS)$ that has a curve whose spine is isomorphic as metric graphs to $G_i$ for each $i$ (here we use that $\SSS$ is full). 
The coordinates of this mesh in $\Hdel(\SSS)$ is a vector in $\RR^d$, where $d=n+g-\dim(\sigma(\SSS))$, by \Cref{cell_dim}.
This vector induces a homomorphism $\psi_{\Mesh'}: \ZZ^d \to \RR_{\ge0}$.

The Harnack mesh $\Mesh'$ is completely determined by $\phi_{\Mesh'}$ and $\psi_{\Mesh'}$, 
so we obtain an embedding from a neighborhood of $\Mesh$ in $\Hdelc$ to $\homm(\LL_\SSS \times \ZZ^d, \RR_{\ge 0})$ given by
\[\Mesh'\mapsto \left((x,y) \mapsto e^{-\phi_{\Mesh'}(x)}\psi_{\Mesh'}(y)\right)\]
where $x\in \LL_\SSS$ and $y\in\ZZ^d$. 
Since $\homm(\LL_\SSS \times \ZZ^d, \RR_{\ge 0})$ is of the same dimension as $\Hdelc$, this mapping forms a $g$-chart. 

That this $g$-chart is compatible with the $g$-chart of $\Hdel$ given by \Cref{modulitheorem} is a consequence of \Cref{prop_ronkin}. Similarly, charts constructed this way are compatible with each other.


\end{proof}

\begin{corollary}
\Cref{conj:local} holds when $\SSS$ is full.
\end{corollary}
\begin{proof}
For any $\LL$, the space $\homm(\LL, \RR_{\ge 0})$ is stratified by its support. 
All of the strata are again of the form $\homm(\LL', \RR_{\ge 0})$ for some submonoid $\LL'\le \LL$ and are topological manifolds with boundary.  
The $g$-charts constructed above respects the cell strata of $\homm(\LL_\SSS\times \ZZ^d, \RR_{\ge 0})$ and $\Hdelc$, so the result follows. 
\end{proof}


Unfortunately, we do not know of a good way of constructing $g$-charts for points in cells corresponding to non-full subdivisions.
Since the edges of cycles corresponding to ovals contracting to a point have finite length, a chart the length of these edges would not have an open set as preimage.
One could expect a $g$-chart covering $\Hdel(\SSS)$ for a non-full subdivision $\SSS$ to be from an open subset of $\homm(\LL_\SSS\times\ZZ^d\times\NN^k, \RR_{\ge 0})$ where $k$ is the number of missing points of $\SSS$.
However, it is not clear what the coordinates corresponding to the copies of $\NN$ should be.

For example, using the area of ovals as coordinates, as in \Cref{modulitheorem}, does not work.
Consider the bottom right corner of \Cref{ex:hdelc}.
If we take a continuous path along the interior of the triangle by stretching the amoeba horizontally but maintaining the area of the bounded component is constant. 
Since the square bounded by the expanded spine is contained in the union of the amoeba with the bounded component of the complement, its area is bounded. 
Stretching the amoeba horizontally causes the length of the vertical edges of the square to tend to 0, which means that the path ends in the bottom right corner. 
This coordinate should be $0$ at this point, so the continuity is broken. 

\begin{question}
Is there a natural way of completing a $g$-atlas on $\Hdelc$ with $g$-charts respecting the cell strata?
\end{question}

A positive answer to this question implies a positive answer to the conjectures in \Cref{subsec:CW}.

\subsection{A cell complex for $T$-curves.} 

Harnack meshes can also be patchworked into non-Harnack curves by choosing polynomials with different sign patterns. The resulting curves are $T$-curves and conversely any $T$-curve arises this way. They can be thought of as the `neighborhood' of $\Hdelc$, which suggests the following question. 
\begin{question}
\label{quest_complex}
Given a lattice polygon $\Delta$, are there other topological types of curves in $X_\Delta$ such that their moduli space can be given a cell complex structure similar to $\Hdelc$? Can they be glued together to form a cell complex, or even a polytopal complex, where cells correspond to different topological types?
\end{question}

\begin{example}
When $\Delta$ is the unit square, $\Hdelc$ is a segment. When the Harnack meshes of the extremes are patchworked in a non Harnack way, we get a curve whose amoeba has a pinching (see \cite[Example 1]{mikhalkin2000real}). From one of the extremes the resulting expanded spine has a bounded edge parallel to $\{x_1=x_2\}$ and from the other extreme the edge is parallel to $\{x_1= -x_2\}$. When the length of the bounded edge goes to 0, both cases degenerate to a reducible curve (the union of two axis-parallel lines). In this case the complex of Question \ref{quest_complex} exists and it is isomorphic to the boundary of a triangle. 
\end{example}

\addtocontents{toc}{\protect\vspace*{\baselineskip}}

\addcontentsline{toc}{chapter}{Bibliography} 
\nocite{*} 
\bibliographystyle{alpha} 
\bibliography{biblomaestria} 



\end{document}